\documentclass[10pt]{article}  % include bezier curves
\usepackage{xcolor}              % Need the color package
\usepackage{graphicx}
\usepackage{amsmath,amssymb, amsfonts, amsthm}
\usepackage{latexsym}
\usepackage{bbm}
\usepackage{hyperref}

\usepackage[version=3]{mhchem}

\usepackage[left=1in,right=1in,top=1in,bottom=1in]{geometry}

\definecolor{MyDarkBlue}{rgb}{0,0.08,0.50}
\definecolor{BrickRed}{rgb}{0.65,0.08,0}

\hypersetup{
colorlinks=true,       % false: boxed links; true: colored links
    linkcolor=MyDarkBlue,          % color of internal links
    citecolor=BrickRed,        % color of links to bibliography
    filecolor=red,      % color of file links
    urlcolor=cyan           % color of external links
}

\usepackage[shortlabels]{enumitem}
\usepackage[mathscr]{eucal}
\usepackage{psfrag}
\usepackage{epsfig}
\usepackage{tikz}
\usetikzlibrary{automata}

\usepackage[numbers,sort]{natbib}

\usepackage{relsize}

\newtheorem{Theorem}{Theorem}[section]
\newtheorem{Lemma}[Theorem]{Lemma}
\newtheorem{Proposition}[Theorem]{Proposition}
\newtheorem{Corollary}[Theorem]{Corollary}

\theoremstyle{definition}
\newtheorem{Example}{Example}[section]
\newtheorem{Definition}{Definition}[section]
\newtheorem{remark}{Remark}[section]
\newtheorem{Assumption}{Assumption}

\newcommand{\bigCI}{\mathop{\underline{\raisebox{0pt}[0pt][1pt]{$\;||\;$}}}}

\newcommand{\s}{\quad}
\newcommand{\non}{\nonumber}

\newcommand{\clg}{\mathcal{G}}

\newcommand{\cls}{\mathcal{S}}

\newcommand{\clr}{\mathcal{R}}
\newcommand{\cla}{\mathcal{A}}

\newcommand{\Z}{\mathbb{Z}}
\newcommand{\R}{\mathbb{R}}

\newcommand{\prPhi}{\{\Phi(t)\colon t\ge 0\}}
\newcommand{\prX}{\{X(t)\colon t\ge 0\}}
\newcommand{\prXuptot}{\{X(s)\colon 0\leq s\leq t\}}
\newcommand{\prXp}{\{X'(t)\colon t\ge 0\}}
\newcommand{\prZ}{\{Z(t)\colon t\ge 0\}}
\newcommand{\prXZ}{\{(X(t),Z(t))\colon t\ge 0\}}
\newcommand{\pFt}{\{\mathcal{F}^X_t\colon t\ge 0\}}
\newcommand{\prW}{\{W(t)\colon t\ge 0\}}
\newcommand{\prZp}{\{Z'(t)\colon t\ge 0\}}
\newcommand{\prXZp}{\{(X'(t),Z'(t))\colon t \ge 0\}}

\newcommand{\Fs}{\mathcal{F}^{X}_{s}}
\newcommand{\Ft}{\mathcal{F}^{X}_{t}}

\newcommand{\todist}[1]{\xrightarrow[#1]{\mathcal{L}}}

\newcommand{\transf}{\rightsquigarrow}
\newcommand{\ntransf}{\not\rightsquigarrow}

\newcommand{\floor}[1]{\left\lfloor #1 \right\rfloor}

\newcommand{\conv}{\ast}
\DeclareMathOperator*{\Conv}{\scalebox{1.5}{\raisebox{-0.2ex}{$\circledast$}}}

\numberwithin{equation}{section}

\usepackage{soul}

\begin{document}

\tikzset{every node/.style={auto}}
 \tikzset{every state/.style={rectangle, minimum size=0pt, draw=none, font=\normalsize}}
 \tikzset{bend angle=20}

	\author{Daniele Cappelletti, 
       \thanks{ETH Zurich, daniele.cappelletti@bsse.ethz.ch
	}
\and Abhishek Pal Majumder
	\thanks{Stockholm University, majumder@math.su.se
	}
	\and
	Carsten Wiuf
	\thanks{University of Copenhagen, wiuf@math.ku.dk}\,\,\thanks{Corresponding author. Department of Mathematical Sciences, Universitetsparken 5, 2100 Copenhagen, Denmark
	}
      }

\title{{\bf Long-time asymptotics of  stochastic reaction systems}}

\maketitle


\bigskip

\begin{abstract}
We study the stochastic dynamics of a system of interacting species in a stochastic environment by means of a  continuous-time Markov chain with transition rates depending on the state of the environment. Models of gene regulation in systems biology take this form. We characterise the finite-time distribution of the Markov chain, provide conditions for ergodicity, and  characterise the stationary distribution  (when it exists) as a mixture of Poisson distributions. The mixture measure is uniquely identified as the law of a fixed point of a stochastic recurrence equation. This recursion  is crucial for statistical computation of moments and  other distributional features.
\end{abstract}

\section{Introduction} 

Reaction networks are widely used in system biology to describe the evolution of interacting molecular species. Though  the nature of interaction is motivated by examples of (bio)chemical reactions, similar models are considered in 
genetics \cite{E79}, epidemiology \cite{PCMV15}, and ecology \cite{G83}.
In recent years, stochastic models of reaction networks, based on continuous time Markov chains (CTMCs), have become fashionable, as the counts of molecular species    in experimental settings might be low or fluctuate considerably over time. In addition to this, a system of interacting species might itself be embedded into a stochastic environment that affects the reaction propensities   and the availability of resources.

A particular focus of stochastic reaction network theory has been to understand the long term behaviour of a model. For a class of reaction networks that has a `complex balanced' equilibrium in a deterministic sense, the stationary distribution has a Poisson product-form \cite{anderson2010product,cappelletti2016product,anderson2018non}, akin to  results in queuing theory.   Another important class of reaction networks, closely related to  complex balanced networks and to the reaction networks considered in this paper, is that of mono-molecular reaction networks  \cite{jahnke2007solving,gadgil2005}. For this class, not only the stationary distribution but also the finite-time distributions can be determined. Specifically, the latter  takes the form of a convolution of  Multinomial and Poisson distributions with parameters evolving according to an ODE system. The equilibrium of the ODE system determines the parameters of the stationary distribution. In these cases, the environment is constant, that is,  absent.

In this work we study the evolution of a stochastic mono-molecular reaction network in a stochastic environment. In the biological context the stochastic environment might itself be a reaction network, however, we assume a more  general setting. As an example, consider the reaction network
$$
0 \ce{<=>} E_{1}, \quad 0 \ce{<=>} E_{2}, \quad 
 E_1 \ce{->} E_1+ S, \quad  E_2+S \ce{->} E_2.
$$
The first two reactions form a fluctuating stochastic environment, where the molecules (resources) $E_1$ and $E_2$ continuously are produced and degraded. If $E_1$ and/or  $E_2$ are present they catalyse the production of the substrate   $S$ and its subsequent degradation. The availability of $E_1,E_2$ is determined by the first two reversible reactions which occur independently of the number of substrate molecules $S$. The production rate of $S$ depends linearly on the number of $E_1$ molecules, while the degradation of $S$ depends linearly on the number of $E_2$ molecules as well as the number of substrate molecules. Biologically, it is an example  of a non-mutual symbiotic relationship (parasitism) where the presence of the parasite species (here $S$) depends on the presence of the host species (here $E_1,E_2$), but not vice versa.

We interpret a stochastic reaction network in a stochastic environment (to be defined in Section \ref{Backg}) as a Markov process $\prXZ$ on a joint state space $\Gamma\times \Z_{\ge 0}^d$. The marginal process $\prX$ is itself a Markov process on the state space $\Gamma$ and accounts for the environments.
The process counting the species of the reaction network is $\prZ$ with state space  $\Z_{\ge 0}^d$.   Recent work has been done to understand (exponential) ergodicity in similar, though different, settings \cite{cloez2015exponential,shao2014ergodicity,shao2015ergodicity}, where the environmental process is  a switching process between regimes.

 In the setting of a  stochastic reaction network in a stochastic environment, we derive the finite-time distributions of $Z(t)$, conditional on the trajectories  $X(s)$, $0\le s\le t$. Specifically, we show that the finite-time distributions are obtained as a convolution of multinomial and Poisson distributions, similar to the mono-molecular reaction networks in a constant environment discussed above. The particular difference being that the parameters depend on the trajectories $X(s)$, $0\le s\le t$, in an accumulative way (not just on $X(t)$), unlike the parameters of the mono-molecular networks in a constant environment, where the parameters solve an ODE. Furthermore, we give conditions under which the the joint process $\prXZ$ is ergodic in terms of ergodicity of the environmental process $\prX$ and conditions on the count process $\prZ$. 
 In recent work \cite{gupta2014scalable}, an approach based on linear Lyapunov functions provides  sufficient conditions for the existence and uniqueness of a stationary distribution of  large class of stochastic reaction networks. It is noteworthy, that these results  remain inconclusive for simple examples  in our setting.

An explicit characterisation of the stationary distribution is however not available. In Theorem \ref{T1} we interpret the stationary distribution of the mono-molecular species  as a mixture of Poisson distributions, thus providing an example of a Poisson representation \cite{G83}. We show that the mixing measure appears as the solution of a random affine equation (also called a stochastic recurrence equation \cite{buraczewski2016stochastic}). Stochastic recurrence equations have been studied in other contexts, see for example  \cite{bertoin2005exponential}. In our case, it is quite remarkable that the randomness of the environment is reflected in the solution of the stochastic recurrence equation  only through  path-wise functionals (for simple cases they are of integral forms). We demonstrate the usefulness of the stochastic recurrence equation by providing a  simulation scheme to simulate from the stationary distribution of the process $\prZ$ given the state of the environment.

 This paper is organised as follows. In Section \ref{Backg} we provide background and notation, and introduce the processes  we are interested in formally.  In Section  \ref{Ex}  an example is given, where  the exact finite-time distribution of the counting process $\prZ$ can be computed. We  generalise the example  in Section \ref{finite time} and discuss the long time behaviour is  in Section \ref{StatDist}. Examples are given in Section \ref{regulatory}, and  a simulation scheme using  the invariant measure  is provided in Section \ref{sample}.

\section{Background and definitions}\label{Backg}

In the following  $\Z$, $\Z_{\geq0}$, and $\Z_{>0}$ denote the integers, the non-negative integers, and the positive integers, respectively. Similarly, $\R$ and $\R_{\ge 0}$ denote the real and the non-negative real numbers. For any $u\in \R$, $\floor{u}=\max\{m\in\Z\colon m\leq u\}$ denotes the largest integer smaller than or equal to $u$. We denote the $i$-th  unit vector in $\R^n$ by $e_i$ and let $e=\sum_{i=1}^{n}e_i$ be the vector whose entries are all one. If $u,v\in\R^n$, we write $u\leq v$ or $u\geq v$ if the inequality holds component-wise. Moreover, define
$$u!=\prod_{i=1}^n u_i!\quad\text{and}\quad u^v=\prod_{i=1}^n u_i^{v_i}.$$
The $n\times n$ identity matrix is denoted by $I_n$. If $A=(a_{ij})_{i,j}$ is a real $n_1\times n_2$ matrix,  then $A^\top$ denotes the transpose  matrix, and $\|A\|_{1}=\max_{j=1,\ldots,n_2}\sum_{i=1}^{n_1} |a_{ij}|$ denotes the $L_{1}$-matrix norm of $A$. In particular, if $v\in\R^n$ then $\|v\|_{1}=\sum_{i=1}^n |v_i|$. For square matrices we define
 $$\prod_{i=m_1}^{m_2}A_i=\begin{cases}
                           A_{m_2}A_{m_2-1}\ldots A_{m_1} &\text{if }m_1\le m_2,\\
                           A_{m_2}A_{m_2+1}\ldots A_{m_1} &\text{otherwise}.
                          \end{cases}
$$
The extremes $m_1$ and $m_2$ are potentially infinite. Note the matrices are ordered from high to low index.

 For a random variable $W$, $\mathcal{L}(W)$ denotes the law of $W$. Given two random variables $W_1$ and $W_2$, defined on the same probability space, we write $W_1\bigCI W_2$ if they are independent, and 
 $W_1\sim W_2$ if 
 $\mathcal{L}(W_1)=\mathcal{L}(W_2)$.
  Given two probability distributions $\mathcal{L}_1$ and 
 $\mathcal{L}_2$ on $\R^n$ (or a subset thereof), their convolution is denoted by $\mathcal{L}_1\conv \mathcal{L}_2$. 
 Let
 $$\Conv_{i=1}^n \mathcal{L}_i=\mathcal{L}_1\conv\cdots\conv\mathcal{L}_n.$$
 
 Convergence in distribution is denoted $W_{i}\todist{} W$ for $i\to\infty$, where $i\in\Z_{\ge 0}$ or $i\in\R_{\ge 0}$
 A stochastic process $\prW$ with values in a metric space $H$ is  tight  if for  $\varepsilon>0$ there exists a compact set $M_{\varepsilon}\subseteq H$ such that $\sup_{t\in\R_{\geq0}}P(W(t)\notin M_{\varepsilon})\le \varepsilon$. 
 A continuity set $A$ of a random variable $W$ is a measurable set such that $P(W\in \partial A)=0$, where $\partial A$ denotes the boundary of the set $A$.

Let $P_{t}(\cdot,\cdot)\colon\cla\times \mathcal{F}\to [0,1]$, $t>0$, be the Markov transition kernel of a  CTMC $\prX$ on a measure space $(\cla,\mathcal{F})$. If all  states communicate with each other,  the chain is said to be irreducible. A stationary distribution for $\prX$ is a probability measure $\pi$ on $(\cla,\mathcal{F})$ satisfying $\int_\mathcal{A}\int_B P_{t}(x,dy)d\pi(x)=\pi(B)$ for all $B\in\mathcal{F}$ and $t>0$. The process $\prX$ is said to be \emph{ergodic for any initial condition} if for any $x\in\mathcal{A}$  there exists a probability measure $\pi_x$ on $(\cla,\mathcal{F})$ such that
$\pi_x(A)=\lim_{t\to\infty}P_{t}(x,A)$ for all $A\in\mathcal{F}$. The probability measures $\pi_x$ are necessarily stationary distributions.  If $\pi_x$ is independent of $x$ then the process is said to be \emph{ergodic}. If the state space is irreducible and the process is ergodic then $\pi_x$ is the unique stationary distribution on $(\cla,\mathcal{F})$.

Next we introduce  notation for probability distributions and  random variables that will appear in various results. $\textup{Exp}(\mu)$ with $\mu>0$  denotes an exponential random variable  with mean $\mu$. For   $m,n \in\Z_{\geq0},$ and $ p=(p_{1},\ldots,p_n)\in[0,1]^n$  with $\sum_{j=1}^n p_j\le 1,$  $\textup{Multi}(m,p)$  denotes an $\{x\in\Z_{\geq0}^n \colon\sum_{i=1}^n x_{i}\le m\}$-valued random variable such that 
\begin{align*}
P(\textup{Multi}(m,p)=(i_1,\ldots,i_n)) &= \frac{m!}{i_{1}!\ldots i_n!(m-\sum_{j=1}^n i_j)!} p^{i_1}_1\ldots p^{i_n}_n\Big(1-\sum_{j=1}^n p_{j}\Big)^{m-\sum_{j=1}^n i_j} \\
&=\binom{m}{i_1,\ldots,i_n} p^{i_1}_1\ldots p^{i_n}_n\Big(1-\sum_{j=1}^n p_{j}\Big)^{m-\sum_{j=1}^n i_j}.
\end{align*}
Note that the latter is not the usual notation for a multinomial random variable and the multinomial coefficient.
In particular, if $n=1$, then $\textup{Multi}(m,p)$   is a binomial random variable, $\textup{Bin}(m,p)$.
 For  $n\in \Z_{\geq0}$, $\mu=(\mu_1,\ldots,\mu_n)\in\R^n_{\geq0},$ we let  $\textup{Pois}(\mu)$ denote an $\Z_{\geq0}^n$-valued random variable such that
$$P\big(\textup{Pois}(\mu)=(i_1,\ldots,i_n)\big)=\prod_{j=1}^n e^{-\mu_{j}}\frac{\mu^{i_j}_j}{i_j!}.$$ 
Namely, $\textup{Pois}(\mu)$ is distributed as $n$ independent Poisson variables with non-negative rates.

\subsection{Stochastic reaction  systems with  stochastic environments}

A reaction network  with species set $\cls=\{S_1,\ldots,S_d\}$ is of a set of reactions $\clr=\{y_1\to y_1',y_2\to y_2',\ldots y_k\to y_k'\}$, where $y_r=\sum_{j=1}^d \gamma_{rj} S_j,$ $y_r'=\sum_{j=1}^d \gamma'_{rj} S_j$ are linear  non-negative integer combinations of species. 
The left hand side of a reaction is called the \textit{reactant}, the right hand side, the \textit{product},  jointly they are \textit{complexes}, and $\xi_r = y'_r- y_r\in\R^d$ is the \emph{reaction vector}, the net gain of species in a reaction. We assume that any species takes part in at least one reaction and that the reactant and product sides are never identical. 
 
A reaction network can be defined by its \emph{reaction graph}, a directed graph with node set the complexes and edge set the reactions between complexes.

The evolution of the species counts  is usually modelled as a homogeneous CTMC $\prZ$ with state space $\Z_{\geq0}^d$ \cite{AK:book, ET:book}. Here, we consider a generalisation of the standard setting, assuming the transition rates of $\prZ$ change over time in a stochastic way.

\begin{Definition}
 A stochastic reaction system with stochastic environment is a triple $\left(\clg, \Lambda, \prX\right)$ such that:
 \begin{itemize}
  \item $\clg$ is a reaction network,
  \item $\prX$ is a homogeneous CTMC with irreducible discrete state space $\Gamma$. 
  Furthermore, we assume $\prX$ is \emph{non-explosive} (or \emph{regular}), implying that  the number of jumps in a bounded time interval  is almost surely (a.s.) finite \cite{norris1998markov}. 
   \item $\Lambda=(\lambda_1,\ldots,\lambda_k)$ is a vector of $k$ \emph{reaction rate functions}, referred to as a \emph{kinetics}, such that  $\lambda_{r}\colon\Gamma\times\Z_{\geq0}^d \to \R_{\ge 0}$ and $\lambda_r\not=0$ a.s., $r=1,\ldots,k$.
  \end{itemize}
\end{Definition}

The evolution of the species counts of $\cls$ is modelled  by the stochastic process $\prZ$, such that $\prXZ$ is a homogeneous CTMC with state space $\Gamma\times\Z_{\geq0}^d$, satisfying the following:
\begin{enumerate}
 \item The processes $\prX$ and $\prZ$ never jump at the same time a.s.
 \item For  $x\in\Gamma$, $z,z+\xi\in\Z_{\ge 0}^d$,  the transition rate from $(x,z)$ to $(x,z+\xi)$ is 
  $$\sum_{\substack{r\colon y_r\to y'_r\in\clr \\ y'_r-y_r=\xi}}\lambda_{r}(x,z).$$
 \item\label{part:X_ind} For  $x,x'\in \Gamma$ and  $z\in\Z_{\geq0}^d$, the transition rate from $(x,z)$ to $(x',z)$ equals the transition rate of $\prX$ from $x$ to $x'$, and does not depend on $z$.
 \item $\lambda_{r}(x,z)> 0$ only if $z\ge y_r$. This avoids $\prZ$ from exiting the state space $\Z_{\geq0}^d$.
\end{enumerate}
By definition $\prX$ and $\prXZ$ are CTMCs, but not necessarily $\prZ$. Note that the Markov property of $\prX$ is in accordance with \eqref{part:X_ind}.

   If the rate functions $\lambda_r$ are independent of the stochastic environment $\prX$, that is, $\lambda_r(x,z)=\lambda_r(z)$ for $1\leq r\leq k$, then we simply refer to $(\clg,\Lambda)$ as a stochastic reaction system. This definition coincides with standard terminology \cite{AK:book, ET:book}.

Let $(\Fs)_{s\geq0}$ denote the filtration of $\prX$. Then, it follows from the above definition that for  $0<s\leq t$ the random variable $X(t)$ is conditionally independent of $Z(s)$ given $\Fs$, that is,
\begin{equation}\label{parasite}
X(t) \bigCI Z(s)\,\Big|\, \Fs \quad \text{for}\quad s<t.
\end{equation}
The property \eqref{parasite} can be recast as
$\mathcal{L}(X(t)\big|X(s),Z(s))=\mathcal{L}(X(t)\big| X(s))$. Models with a similar conditional structure is studied in \cite{bowsher2010stochastic}.
 
 Following  \cite{kurtz1972relationship}, we might write
\begin{equation*}
Z(t)=Z(0)+\sum_{r=1}^{k}N_{r}\!\left(\int_{0}^{t}\lambda_{r}\left(X(s),Z(s)\right)ds\right)\xi_{r},
\end{equation*}
where $\{N_{r}\colon r=1,\ldots,k\}$ is a set of i.i.d.\ unit-rate Poisson processes and $Z(0)$ is the initial state at time zero.  An equivalent description of $\prZ$ is the following: Conditioned on the path of $\prX$, $X(s)=x(s)$ for all $s\in[0,t]$, the process 
$\prZ$ is a non-homogeneous CTMC
 \begin{equation}\label{eq:plugging_in}
  Z(t)\,\,\sim \,\,Z(0)+\sum_{r=1}^{k}N_{r}\!\left(\int_{0}^{t}\lambda_{r}(x(s),Z(s))\,ds\right)\xi_{r}.
  \end{equation}

A common choice of kinetics is  stochastic mass-action kinetics, which in our setting takes the form
\begin{equation*}
\lambda_{r}(x,z)= \kappa_{r}(x)\frac{z!}{(z-y_{r})!}1_{\{x\in\Z^d_{\ge0}\colon x\geq y_{r}\}}(z),\quad r=1,\ldots,k, 
\end{equation*}
 for some functions $\kappa_{r}\colon \Gamma\to \R_{\geq0}$, $r=1,\ldots,k$.  
 Mass-action kinetics corresponds to the hypothesis that the system is well-stirred such that the propensity of a reaction is proportional to the number of  combinations of the  species in the reactant. A stochastic reaction system (with stochastic environment)  with mass-action kinetics is called a stochastic mass-action system (with stochastic environment).

\begin{Example}\label{ex:typical}
 A typical biological situation is that of a system of non-mutual symbiotic interactions where some species  evolve conditionally on the presence of other species, the stochastic environment. As an example, consider the reaction network
\begin{align*}
 0 \ce{<=>[\lambda_1][\lambda_2]} S_{1}, &\qquad 0  \ce{<=>[\lambda_3][\lambda_4]} S_{2},\\
 S_{2} \ce{->[\lambda_5]} S_{2}+m S_3, &\qquad  S_{1}+S_3  \ce{->[\lambda_6]} S_{1},
\end{align*}
where $m$ is an integer. The set of reactions is enumerated according to the index of $\lambda_r$, written above (or below) the arrow of $y_r\to y_r'$. Assume the species counts evolve according to a stochastic mass-action system with constant environment, that is, with transition rates
$$\begin{array}{lll}
   \lambda_1(w)=\kappa_1, & \lambda_2(w)=\kappa_2w_1, & \lambda_3(w)=\kappa_3,\\
   \lambda_4(w)=\kappa_4w_2, & \lambda_5(w)=\kappa_5w_2, & \lambda_6(w)=\kappa_6w_1w_3,
  \end{array}$$
with $\kappa_r>0$ and  $w=(w_1,w_2,w_3)\in\Z_{\geq0}^3$.  Let $\{Y(t)\colon t\ge 0\}$ be the associated CTMC.

The number of  $S_3$ molecules   does not affect the numbers of  $S_1$ and $S_2$  molecules. Therefore,  we might regard the latter as constituting a stochastic environment, which affects the production and degradation of $S_3$. Namely,  consider $\prX=(Y_1(\cdot), Y_2(\cdot))$ with state space $\Gamma=\Z_{\geq0}^2$ and $\prZ=Y_3(\cdot)$. In this case, the reaction graph associated with $\prZ$ is simply
$$0 \ce{->[\widetilde\lambda_1]} m S_3, \qquad  S_3  \ce{->[\widetilde\lambda_2]} 0,\qquad \widetilde\lambda_1(x,z)=\kappa_5x_2,\qquad \widetilde\lambda_2(x,z)=\kappa_6x_1z.$$

While this is a case  arising in practice, for example, gene (de)activation controlling the production of proteins, more general reaction rates might be considered. For example, 
$$\widetilde\lambda_1(x,z)=\frac{\kappa_5}{1+x_2},\qquad \widetilde\lambda_2(x,z)=\frac{\kappa_6}{1+x_1^3}z.$$
Moreover, the process $\prX$ does not have to be confined to non-negative vectors: rescaled species counts could be considered in the spirit of multiscale analysis \cite{KangKurtz} or even negative states of $\prX$ could be considered, for example, modelling the effect of temperature on  the transition rates of $\prZ$.
\end{Example}

\begin{remark}\label{rem:independence}
  Consider a species $S$ such that the reactions consuming it are of the form $S\to y'$, where $y'$ is a complex not involving $S$. Under the assumption of mass-action kinetics, one might construct the process $\prZ$ such that the fate of the present molecules of the species $S$  (what they are transformed into and when) are  conditionally independent given $\{\Ft\colon t\ge 0\}$. In fact,  the reaction rate function of $S\to y'$ is of the form $\lambda(X(t),z)=\kappa(X(t)) z_i$ (so $S$ is the $i$-th species). This rate might be constructed from $z_i$ independent exponentially distributed random variables with rate $\kappa(X(t))$.
\end{remark}

\section{A case study}\label{Ex}

 We begin our formal analysis with an example, essentially expanding Example \ref{ex:typical}.  We show that the distribution of $Z(t)$ can be explicitly computed in terms of the path $\{X(s)\colon 0\leq s\leq t\}$. This can be accomplished in general assuming  the reaction rate functions are linear in $Z(t)$, see Section \ref{mod}.

 Consider a stochastic reaction system, $\left(\clg, \Lambda, \prX\right)$, 
$$0 \ce{->[\lambda_1]} m S, \qquad  S \ce{->[\lambda_2]} 0,\qquad \lambda_1(x,z)=\kappa_1(x),\qquad \lambda_2(x,z)=\kappa_2(x)z,$$
for fixed $m\in\Z_{\ge 0}$.  The distribution of $Z(t)$ given $X(s)$, $0\le s\le t$, is a convolution of  simple distributions.

\begin{Proposition}\label{P1}  For the model described above, the conditional distribution of $Z(t)$ given $\Ft$, for $t>0$,  is
 $$ \mathcal{L}\left(\left.\textup{Bin}\left(Z(0), e^{- \int_{0}^{t}  \kappa_2(X(s))ds}\right)\right|\Ft\right) \conv \left(\Conv_{i=1}^m \mathcal{L}(iN_i|\Ft)\right),$$
where $N_{i}$, $i=1,\ldots,m$, are Poisson random variables with mean  
\begin{equation}\label{poissonmean}
 \binom{m}{i} \int_{0}^{t}  \kappa_1(X(u)) e^{-i\! \int_{u}^{t}\kappa_2(X(s))ds}(1-  e^{-\!\int_{u}^{t}\kappa_2(X(s))ds})^{m-i}\,du, \quad i=1,\ldots,m.
\end{equation}
\end{Proposition}

\begin{proof} We begin with the case $m=1$. We show that $\mathcal{L}\left(Z(t)\big| \Ft\right)$ is 
\begin{equation}
\mathcal{L}\left(\left.\text{Bin}\left(Z(0), e^{- \int_{0}^{t}  \kappa_2(X(s))ds}\right)\right|\Ft\right) \conv \mathcal{L}\left(\left.\text{Pois}\left(\int_{0}^{t}  \kappa_1(X(u))e^{- \int_u^{t}  \kappa_2(X(s))ds} \,du\right)\right|\Ft\right). \label{eq1}
\end{equation}

Conditioned on $\Ft$, the process $Z(s)$, $0\le s\le t$, evolves as a non-homogeneous CTMC in the sense of \eqref{eq:plugging_in}. Moreover, the degradation events of the $S$ molecules  are conditionally independent of each other, see   Remark \ref{rem:independence}. 
Specifically, the time until degradation of an $S$ molecule  is exponentially distributed with time dependent rate. Hence, the number of $S$ molecules that survived until time $t$ starting from $Z(0)$ molecules at time $0$ is distributed as 
$$ \text{Bin}\left(Z(0), e^{- \int_{0}^{t}  \kappa_2(X(s))ds}\right)$$
given $\Ft$, which accounts for the first term in \eqref{eq1}.
 We will next show that given $\Ft$,  the total number of surviving $S$ molecules at time $t$ out of those born  in $(0,t]$ is distributed as
 $$\text{Pois}\left(\int_{0}^{t} \kappa_1(X(u))e^{- \int_{u}^{t} \kappa_2(X(s))ds} \,du\right)$$
 given $\Ft$. Let the total number of birth instances of $S$ molecules in $[0,t]$ be $B_{t}$, which follows a Poisson distribution with  mean 
 $\lambda_t= \int_{0}^{t} \kappa_1(X(u)) \,du$,  conditioned on $\Ft$. Birth instances conditioned on $B_t$ can be thought of as independent realisations from the density 
\begin{equation}
\frac{  \kappa_1(X(\cdot)) } {\int_{0}^{t} \kappa_1(X(u)) \,du}.\label{birthdensity}
\end{equation}
Each of these  particles survive with  probability
 $$P_{t}=\frac{\int_{0}^{t}\kappa_1(X(u))e^{-\int_{u}^{t}\kappa_2(X(s))ds}\,du}{\int_{0}^{t} \kappa_1(X(u))\,du},$$
 where the exponential term is the probability of survival given a molecule is born a time $u$. Multiplying  $P_t$ with $\lambda_t$ proves  \eqref{eq1}.
 
 We generalise  \eqref{eq1} to arbitrary  $m>1.$ The fate of the initial $Z(0)$ molecules is described similarly to the case $m=1$. For each birth occurrence (that is, a firing of the reaction $0\to mS$) there are $m$ $S$ molecules being born and among them either 0, 1, $\ldots$  or $m$ $S$ molecules will survive until time $t$. Given $\Ft$, the probability that exactly $i$ out of the $m$ $S$ molecules that are born at time $u\in(0,t]$ survive until time $t$ is
\begin{equation*}
P^{t}_{u}(i)= \binom{m}{i} e^{-i\int_{u}^{t}\kappa_2(X(s))ds}\left(1-  e^{-\int_{u}^{t}\kappa_2(X(s))ds}\right)^{m-i}.
\end{equation*}

Let the unordered birth instances of $S$ molecules in $(0,t]$ be $U=(U_{1},\ldots,U_{B_t})$ where each $U_i$ is a realisation from \eqref{birthdensity}. Define   $Y(t)=\sum_{i=1}^{B_t}Y_{U_{i}},$ where  $Y_{U_{i}}= j,$ if $j\in\{0,\ldots,m\}$ molecules survived at time $t$ for the birth instance at $U_{i} \in (0,t]$.  It follows that 
 $$E\left(e^{-\gamma Y_{U_{i}}}\big| U=(U_{1},\ldots,U_{r}), B_t=r, \Ft\right)=\sum_{j=0}^{m} P^{t}_{U_{i}}(j)e^{-j\gamma }=1-\sum_{j=1}^{m}P^{t}_{U_{i}}(j)(1-e^{-j\gamma})$$
 for $\gamma>0$.
Hence, the Laplace transform of the surviving $S$ molecules that are born in $(0,t]$ is
\begin{align}
 E\left(e^{- \gamma Y(t)}\big| \Ft\right)&= E\left(E\left(e^{-\gamma Y(t)}\big|U=(U_{1},\ldots,U_{r}),B_t=r\right)\Big| \mathcal{F}_{T}^{X}\right)\non\\
&= \sum_{r=0}^{\infty}\int_{0\le U_{1},\ldots,U_{r}\le t}\Big(\prod_{i=1}^{r} \Big[1-\sum_{j=1}^{m}P^{t}_{U_{i}}(j)(1-e^{-j\gamma})\Big]\kappa_1(X(U_i))\Big) \frac{e^{-  \int_{0}^{t} \kappa_1(X(s)) ds}}{r!}dU_{1}\ldots dU_{r}\non\\
&=\sum_{r=0}^{\infty}\bigg[\int_{0}^{t} \kappa_1(X(u))\,du-\sum_{j=1}^{m}(1-e^{-j\gamma})\!\int_{0}^{t} \kappa_1(X(u))P^{t}_{u}(j)\,du\bigg]^{r} \frac{e^{- \int_{0}^{t} \kappa_1(X(s)) ds}}{r!}\non\\
&= e^{-\sum_{j=1}^{m}(1-e^{-j\gamma})\!\int_{0}^{t} \kappa_1(X(u))P^{t}_{u}(j)\,du}.\label{multi}
\end{align}
So $Y(t) \sim N_{1}+2N_{2}+3N_{3}+\ldots +mN_{m}$, where 
$N_{i}$ is distributed as $\text{Pois}\left(\int_{0}^t \kappa_1(X(u)) P_{u}^{t}(i)\,du\right)$ given $\Ft$, and 
  $$\mathcal{L}\left(Z(t)\big| \Ft\right)=\mathcal{L}\left.\left(\text{Bin}\left(Z(0), e^{- \int_{0}^t  \kappa_1(X(s))ds}\right) +N_{1}+2N_{2}+3N_{3}+\ldots +mN_{m}\right|\Ft\right),$$
which completes  the proof.
\end{proof}

The result of Proposition \ref{P1} can be generalised further by letting $\kappa_1$ and $\kappa_2$ be functions of both the state $x\in\Z_{\geq0}$ and time. As a matter of fact, the proofs in this paper for the results concerning finite time all hold equally for reaction rate functions depending on time directly.
Moreover, a slightly modified result holds for more general distributions of inter-arrival times between the occurrences of the degradation reaction $S\to0$, as long as the $S$ molecules are degraded independently on each other given $\{X(t)\colon  t\ge 0\}$, see Section \ref{mod}. For the more general setting,  let $T_{u}$ be the survival time of an $S$ molecule born at time $u$, and define $P_{u,t}=P(T_{u} > t| \mathcal{F}_t^{X}).$
Then the result of Proposition \ref{P1} holds with \eqref{poissonmean} replaced by 
\begin{equation}\label{eq:P1general}
\binom{m}{i} \int_{0}^{t} \kappa_1(X(u),u) P_{u,t}^{i} (1-P_{u,t})^{m-i}\,du,
\end{equation}
where $\kappa_1$ and $\kappa_2$ are now allowed to dependent on time directly.

\section{General case of study}\label{mod}

The main aim  is to study ergodicity of a stochastic reaction system with stochastic environment, assuming that $\prX$ is an ergodic  process. To do so, we  consider  a specific family of models for which  the stationary distribution can be characterised explicitly. To motivate the necessity of further assumptions, we give  two examples. The first is a transient reaction system, regardless the ergodicity of the process $\prX$.
The second is an example of an explosive  reaction system with stochastic ergodic environment. 

\begin{Example}\label{ex:transient}
 Assume the chain $\prX$ has two states,  denoted by 0 and 1. This is often the case in genetic models, where $X(t)$ denotes whether a gene is active at time $t$ or not. Let $q_{10}$ be the transition rate from 1 to 0 and $q_{01}$ the transition rate from 0 to 1.
 
 Consider the following stochastic mass-action system with stochastic environment:
$$S_1+S_2\ce{->[\lambda_1]}S_3,\qquad S_1+S_3\ce{->[\lambda_2]}S_2,\qquad 0\ce{->[\lambda_3]}S_1,$$
with
$$\lambda_1(x,z)=\kappa_1 xz_1z_2,\quad \lambda_2(x,z)=\kappa_2 (1-x)z_1z_3\quad\text{and}\quad \lambda_2(x,z)=\kappa_3 ,$$
for positive constants $\kappa_1$, $\kappa_2$ and $\kappa_3$. The total amount of molecules of $S_2$ and $S_3$ is conserved. Assume that
$$Z_2(t)+Z_3(t)=b>0.$$
When $X(t)=1$, degradation of an $S_1$ molecule consumes a molecule of $S_2$, which is not recreated because $\lambda_2(1,z)=0$. Hence, when $X(t)=1$ at most $b$ molecules of $S_1$ can be degraded. Similarly, $\lambda_1(0,z)=0$ implies that at most $b$ molecules of $S_1$ can be degraded when $X(t)=0$.
Hence, by using renewal theory \cite{durrett} and by setting up the renewal intervals in between two consecutive visits of $\prX$ to 1, we have that 
$$\lim_{t\to\infty}\frac{Z_1(t)}{t}\geq \kappa_3-2b\frac{q_{10}q_{01}}{q_{10}+q_{01}}.$$
It follows that for any choice of parameters such that the latter is strictly positive, the model is transient.

An interesting observation is the following. Consider the stochastic mass-action system with fixed environment
$$S_1+S_2\ce{->[\widetilde{\lambda}_1]}S_3,\qquad S_1+S_3\ce{->[\widetilde{\lambda}_2]}S_2,\qquad 0\ce{->[\widetilde{\lambda}_3]}S_1,$$
where
$$\widetilde{\lambda}_1(z)=\widetilde{\kappa}_1 z_1z_2,\qquad \widetilde{\lambda}_2(z)=\widetilde{\kappa}_2 z_1z_3,\qquad \widetilde{\lambda}_2(z)=\widetilde{\kappa}_3,$$
with   $\widetilde{\kappa}_1, \widetilde{\kappa}_2, \widetilde{\kappa}_3$ positive. For all parameter values, the function $V(z)=z$ is a Lyapunov function, hence the process is always ergodic. Therefore, the stochastic fluctuations of the environment allow for a behaviour that is not possible in the constant environment. As a consequence, the example  suggests that in general it is not correct to average the characteristic of a random environment over time to study the asymptotic behaviour of a stochastic reaction system with stochastic environment. 
\end{Example}

\begin{Example}\label{ex:explosion}
Consider the following stochastic mass-action system with stochastic environment:
$$2S\ce{<=>[\lambda_1][\lambda_2]}3S,\qquad \lambda_1(x,z)=\kappa_1(x)z(z-1),\quad \lambda_2(x,z)=\kappa_2(x)z(z-1)(z-2).$$
Assume  the process $\prX$ is ergodic, and that there exists a recurrent state $x'$ such that $\kappa_1(x')\neq0$ and $\kappa_2(x')=0$.  In this case, with probability one, $\prX$  hits $x'$ infinitely often. When this occurs, the number of $S$ molecules  evolves according to the stochastic mass-action system with constant environment
 $$2S\ce{->[\widetilde{\lambda}_1]}3S,$$
 and $\widetilde{\lambda}_1(z)=\kappa_1(x')z(z-1)$. Specifically, the species count of  $S$ evolves according to a pure birth process with (almost) quadratic birth rate. It follows that whenever $\prX$ hits $x'$,  $\prZ$ has a positive probability of exploding. Hence it will eventually explode with probability one. 
 \end{Example}

\subsection{The linearity assumption}

We  introduce a family of stochastic reaction systems with stochastic environment.   

\begin{Assumption}\label{as0}{(Linearity assumption)}
 Suppose that $(\clg, \Lambda, \prX)$ is a stochastic mass-action system with stochastic environment, such that  each reaction  takes  one of the following forms:
\begin{equation}\label{model}
\text{Production}:\quad 0\to m_jS_j, \quad
\text{Conversion}:\quad S_i\to S_j,\quad 
\text{Degradation}:\quad S_{i}\to 0,
\end{equation}
for  $m=(m_{1},\ldots,m_{d})\in\Z_{\geq0}^{d}$, and $1\leq i,j\leq d$. For notational convenience, we  set $m_j=0$ if there is not  a reaction $0\to \beta S_j$  for any $\beta>0$.
\end{Assumption}

The stochastic reaction system in Section \ref{Ex} satisfies Assumption \ref{as0}, while those of Example \ref{ex:transient} and Example \ref{ex:explosion} do not. 
Under  Assumption \ref{as0}, let  $\widetilde \lambda_{ij}\colon \Gamma\to\R^d_{\geq0}$, for $0\leq i,j\leq d$, be the functions:
\begin{itemize}
 \item  $\widetilde\lambda_{0j}(x)$ is the reaction rate function of $0\to m_jS_j$ if the latter is in $\clr$, and 0 otherwise.
 \item $\widetilde\lambda_{ij}(x)z_i$ is the reaction rate function of $S_i\to S_j$ if the latter is in $\clr$, and 0 otherwise.
 \item $\widetilde\lambda_{i0}(x)z_i$ is the reaction rate function of $S_i\to 0$ if the latter is in $\clr$, and 0 otherwise.
\end{itemize}
We further define the function $A\colon\Gamma\to\R^{d\times d}$ as $A(\cdot)=(\,a_{ij}(\cdot)\,)_{1\le i,j\le d}$ with
\begin{equation}\label{eq:AX}
 a_{ij}(x)=\begin{cases}
              \widetilde\lambda_{ji}(x)&\text{for }i\neq j,\\
              -\sum_{\substack{k=1 \\ k\neq i}}^{d} \widetilde\lambda_{ki}(x)- \widetilde\lambda_{i0}(x)&\text{for }i=j,
             \end{cases}
\end{equation}
and the function $B\colon\Gamma\to\R^d_{\geq0}$ as $B(\cdot)=(\,b_i(\cdot)\,)_{1\le i\le d}$ with
\begin{equation}\label{eq:BX}
b_i(x)=\widetilde\lambda_{0i}(x).
\end{equation}
For  $x\in\Gamma$, $A(x)$ has non-positive column sums and non-negative off-diagonal elements. Hence it may be considered  the transpose of a sub-generator matrix of a Markov chain.

Given $\pFt$, the functions $A(X(\cdot))$ and $B(X(\cdot))$ are functions of time and to stress this fact we introduce the notation $A_X(t)=A(X(t))$ and $B_X(t)=B(X(t))$ for any $t\ge 0$.

\subsection[The matrix Phi]{The matrix $\Phi(\cdot)$}

Given a function $M\colon\R_{\geq0}\to \R^{n\times n}$ and a non-singular  matrix $H_0\in\R^{n\times n}$, a \textit{fundamental matrix solution} of
\begin{equation}\label{eq:fundamental_matrix}
H(0)=H_0\quad\text{and}\quad\frac{d}{dt}H(t)=M(t)H(t) \quad \text{for all}\quad t> 0,
\end{equation}
is a function $H\colon\R_{\geq0}\to \R^{n\times n}$ that is non-singular for all $t\in\R_{>0}$ and solves \eqref{eq:fundamental_matrix}, where the derivative  is intended component-wise. Such a solution always exists for any non-singular initial condition $H_0$ \cite{perko}.

For our purpose, we define $\prPhi$ as the $\R^{d\times d}$-valued stochastic process that solves
\begin{equation}\label{funda}
\Phi(0)=I_d\quad\text{and}\quad\frac{d}{dt}\Phi(t)=A_X(t)\Phi(t) \quad \text{for all}\quad t> 0.
 \end{equation}
An explicit construction of the matrix $\Phi(t)$ given $\Ft$ is in Lemma \ref{P5.4}.  Intuitively, \eqref{funda} might be thought of as a time-varying Kolmogorov forward equation, as described below. As a consequence, a formal interpretation of the quantity $\Phi(\cdot)$ can be derived, see Lemma~\ref{lem:interpretation}.

 For simplicity, assume   $\prX$ is constant, such that  $A_X=A_X(t)$ is independent of time. In this case, $\prZ$ is a mass-action system in the standard sense: The transition rates are constant  and as a consequence the process is a time-homogeneous CTMC. If Assumption~\ref{as0} holds, then the molecules in the system evolve independently of each other, in the sense of Remark~\ref{rem:independence}. In particular, 
the species type of a specific molecule evolves as a CTMC   with state space $\{S_1,\dots,S_d,0\}$, where $0$ is an absorbing state denoting degradation. The transition matrix for this process is 
 
 $$
\renewcommand\arraystretch{1.3}
Q=\left(
\begin{array}{c|c}
  A^\top_X &  -A^\top_x e \\
  \hline
  0 & 0
\end{array}
\right).
$$
 Hence, if  $P^{(i,j)}(t)$ is the transition probability from state $S_i$ to state $S_j$ over time $t$, then by Kolmogorov backward equation
  $$\frac{d}{dt}P^{(i,j)}(t)= P^{(i,j)}(t)A^\top_X.$$
 It follows that $P^{(i,j)}(t)$ is the $(j,i)$ entry of the matrix $\Phi(t)$. The same holds even if the process $\prX$ changes over time, as stated in Lemma~\ref{lem:interpretation}.

In this spirit  the matrix $\Phi(\cdot)$ has been  used to describe state probabilities for deterministically changing environments \cite{jahnke2007solving}  and for networks of $M/M/\infty$ queues  with time-varying stochastic rates \cite{o1986m,d2008m}, which  \cite{jahnke2007solving} might be considered a special case of. We restate this result in Theorem \ref{thm:jahnke2007solving} below.

For notational convenience, define $\prW$ as the $\R^d_{\geq0}$-valued stochastic process solving
\begin{equation*}
W(0)=0 \quad\text{and}\quad \frac{d}{dt}W(t)=A_{X}(t)W(t)+B_{X}(t) \quad \text{for all}\quad t>0.
\end{equation*}
By Lemma \ref{Appe1} in the Appendix we have
\begin{equation}\label{eq:convergence}
W(t)=\int_{0}^{t}  \Phi(u,t)B_X(u) \,du\quad\text{for all }t\in\R_{\geq0},
\end{equation}
where the integral is intended component-wise, and  
$$\Phi(u,t)=\Phi(t)\Phi^{-1}(u), \qquad \text{for }t\geq u\geq 0.$$
Clearly,  $\Phi(0,t)=\Phi(t)$. Moreover, $\Phi(u,u+\cdot)$ is the fundamental matrix solution solving
\begin{equation}\label{funda_shift}
\Phi(u,u)=I_d \quad\text{and}\quad \frac{d}{dt}\Phi(u,u+t)=A_X(u+s)\Phi(u,u+t) \quad \text{for all}\quad t>0.\\
\end{equation}
Furthermore,  the matrix $\Phi(u,t)$ fulfils the equality
\begin{equation}\label{eq:phi_tu}
\Phi(u,t)=\prod_{i=0}^{k-1} \Phi(t_i,t_{i+1})=\Phi(t_{k-1},t_k)\ldots \Phi(t_0,t_1),
\end{equation}
for any $0\le u= t_0\le t_{1}\le\ldots\le t_k=t$.

\begin{Lemma}\label{P5.4}
Suppose Assumption \ref{as0} holds. Let $T_{0}=0$ and $T_{i}$, $i\ge 1$, denote the  $i$-th jump time of $\prX$. For $s\in[T_{i},T_{i+1})$, let  $A_i=A_{X}(s)$.  Then, the matrix $\Phi(t)$ can be expressed as
\begin{equation}\label{fundjump}
\Phi(t)=e^{A_{N(t)}(t-T_{N(t)})}\prod_{i=0}^{N(t)-1}e^{A_i(T_{i+1}-T_{i})},\quad t\ge 0,
\end{equation}
where $N(t)=\sup\{i\colon T_{i}\le t\}$.  As a consequence,
\begin{equation}\label{phiprod}
\Phi(u,t) =\left\{\begin{array}{ccc} e^{A_{N(t)}(t-T_{N(t)})}\left(\prod_{i=N(u)+1}^{N(t)-1}e^{A_i(T_{i+1}-T_{i})}\right)e^{A_{N(u)}(T_{N(u)+1}-u)} & \textup{for}  & N(u)+1\le N(t),\\
e^{A_{J_t}(t-u)} & \textup{for}& N(u)= N(t). \end{array}  \right.
\end{equation}
\end{Lemma}

\begin{proof}
 Observe that on the interval $[T_{i},T_{i+1})$ the matrix $A_{X}(\cdot)$ is the constant matrix $A_i$ for all $i$. Hence, the fundamental matrix solution to $\Phi'_X(t)=A_{i}\Phi(t)$ for $t\in[T_{i},T_{i+1}]$ is  $\Phi(t) =e^{A_{i}(t-T_{i})}\Phi(T_i).$ The result follows by induction on $i\ge 0$ and by the fact that $\Phi(0)$ is the identity matrix. Equation \eqref{phiprod} follows from $\Phi(u,t)=\Phi(t)\Phi^{-1}_X(u)$ and from \eqref{fundjump} applied to both $\Phi(t)$ and $\Phi^{-1}_X(u)$.
\end{proof}

\section{Finite-time distribution}\label{finite time}

The goal of this section is to describe the distribution of $Z(t)$ given $\Ft$ by means of $\Phi(t)$. We start with a similar result from \cite{jahnke2007solving} when $\prX$ is deterministic.

\begin{Theorem}\label{thm:jahnke2007solving}
 Suppose Assumption~\ref{as0} holds and  that $0\leq m_j\leq 1$ for  $1\leq j\leq d$. Moreover, assume the process $\prX$ is deterministic, that is, there exists a function $x\colon \R_{\geq0}\to\Gamma$ such that 
 $$X(t)=x(t)\quad\text{for all}\quad t\ge 0,$$
 a.s. Then, for  $t\ge 0$,
\begin{equation*}
\mathcal{L}(Z(t))=\Big(\Conv_{i=1}^d\mathcal{L}\left(\textup{Multi}\left(Z_i(0),\Phi(t)e_i\right)\right)\Big)\conv\mathcal{L}(\textup{Pois}\left(W(t)\right)),
\end{equation*}
\end{Theorem}

\begin{Lemma}\label{lem:interpretation}
 Suppose Assumption \ref{as0} holds and let $\Phi(\cdot)$ be as in \eqref{funda}. For  $1\leq i,j\leq d$ and $t\geq u\geq0$, let $P_X^{(i,j)}(u,t)$ be the probability, given $\Ft$, that a molecule of species $S_i$, present at time $u$, eventually is converted to a molecule of species $S_j$ at time $t$. Then, for  $1\leq i,j\leq d$ and $t\geq u\geq0$,
 \begin{enumerate}[(a)]
 \item\label{item:probability} $P_X^{(i,j)}(u,t)$ is  the $(j,i)$-th entry of  $\Phi(u,t)$,
 \item\label{item:degradation_prob} the probability given $\Ft$ that a molecule of species $S_i$, present at time $0$, eventually is degraded by time $t$ is
 $$1-\sum_{j=1}^d P_X^{(i,j)}(u,t)=1-e^\top \Phi(u,t) e_i=1-\|\Phi(u,t) e_i\|_1,$$
 \item\label{item:norm_1_leq_1}  $\|\Phi(u,t) e_i\|_1\leq 1$.
 \end{enumerate}
\end{Lemma}
\begin{proof}
(a) Consider a molecule of species $S_i$ that is present at time $u$, and let $F(s)$ denote its type (a species or $0$ if it is degraded) at time $u+s$, for $s\geq0$. 
 
 Define $\widehat{X}(s)=X(s+u)$ for $s\geq0$. Consider modification of the model where  production is not allowed (that is, $\widetilde{\lambda}_{0i}$ is the null function for all $1\leq i\leq d$). Denote by $\widehat{Z}(\cdot)$ the process associated with this model, and by $\{\widehat{X}(s)\colon s\geq0\}$ the stochastic environment.
 
 The fate of single molecules  are conditionally independent given $\Ft$, in the sense of Remark~\ref{rem:independence}. Then, the process $\{\widehat{Z}(s)\colon s\geq0\}$ given $\widehat{Z}(0)=e_i$ is distributed exactly as the process $\{F(s)\colon s\geq0\}$, and the dependence on the environment is the same. Hence, Given $\Ft$, it follows from \eqref{eq:plugging_in}, Theorem~\ref{thm:jahnke2007solving}, and  \eqref{funda_shift} that for  $s\geq0$ 
\begin{equation*}
\mathcal{L}\Big(\widehat{Z}(s)\big| \Ft, \widehat{Z}(0)=e_i\Big)\mathop{=}\mathcal{L}\left(\text{Multi}(1,\Phi(u,u+s) e_i)\right),
\end{equation*}

(b) This is a straightforward consequence of (a).
(c) This follows from (b).
\end{proof}

The next proposition provides a multi-dimensional version of Proposition \ref{P1}. Compared to the deterministic result in \cite{jahnke2007solving}, we  allow $m_j\ge 0$ to be arbitrary and not restricted to $0,1$.

For any $n_1,n_2\in \Z_{\geq0}$, let $\Theta_{n_1,n_2}$ be the set of $n_2$-tuples  $\{(\nu_{1},\ldots,\nu_{n_{2}})\in \Z_{\geq0}^{n_2} \colon  1\le \sum^{n_2}_{j=1}\nu_{j} \le n_{1}\}$. We refer to $\nu\in \Theta_{n_1,n_2}$ as a configuration.

\begin{Proposition}\label{P2}
Suppose Assumption \ref{as0} holds and let $\Phi(\cdot)$ be as in \eqref{funda}. Then, for any fixed $t\ge 0$, the conditional distribution of  $Z(t)$ given $\Ft$ is
\begin{equation}
\Big(\Conv_{i=1}^d\mathcal{L}\big(\textup{Multi}\left(Z_i(0),\Phi(t)e_i\right)\big| \Ft\big)\Big)\conv\Big(\Conv_{j=1}^d \Conv_{\nu\in \Theta_{m_j,d}} \mathcal{L}\big(\nu\, N_{\nu j}(t) \big| \Ft\big)\Big),\label{multinom1}
\end{equation}
where the initial condition $Z(0)=(Z_1(0),\ldots,Z_{d}(0))$ is the molecular counts at time zero,
$$N_{\nu j}(t)\,\,\sim\,\,\textup{Pois}\left(\int_{0}^{t} \widetilde\lambda_{0j}(X(u))\,g^X_{u,t}(\nu, m_{j})\,du\right),$$
and
$g^X_{u,t}(\nu, m_j)$ is the probability that the multinomial random variable $\textup{Multi}\left(m_j,\Phi(u,t)e_j\right)$ takes the value $\nu$,  given $\Ft$ ($\nu N_{\nu j}$ is the vector $\nu$ multiplied by the number $N_{\nu j}$).

In particular, if $0\le m_j\le 1$ for all $j=1,\ldots,d$, then 
$$\Conv_{j=1}^d \Conv_{\nu\in \Theta_{m_j,d}} \mathcal{L}\big(\nu\, N_{\nu j}(t)\big| \Ft\big)= \mathcal{L}\big(\textup{Pois}(W(t))\big| \Ft\big).$$
\end{Proposition}

\begin{proof}
We begin with the case $0\le m_j\le 1$ for all $j=1,\ldots,d$.  
Given $\Ft$, it follows from \eqref{eq:plugging_in} and Theorem~\ref{thm:jahnke2007solving} that 
\begin{equation}\label{mainJahnke}
\mathcal{L}\big(Z(t)\big| \Ft\big)=\Big(\Conv_{i=1}^d\mathcal{L}\big(\text{Multi}\big(Z_i(0),\Phi(t)e_i\big)\big| \Ft\big)\Big)\ast\mathcal{L}\Big(\text{Pois}(W(t))\big| \Ft\Big).
\end{equation}
If $m_j=0$, then $\Theta_{m_j,d}=\emptyset$. If $m_j=1$, then $\Theta_{m_j,d}=\{ e_i|i=1,\ldots,d\}$  and  $g^X_{u,t}(e_i, m_{j})=e_i^\top\Phi(u,t)e_j$.  It follows that 
$$\sum_{j=1}^d\widetilde\lambda_{0j}(X(u)) g^X_{u,t}(e_i, m_{j})=e^\top_i\Phi(u,t)B_X(u).$$
Using the definition of $W(t)$ in \eqref{eq:convergence} then \eqref{multinom1} is the same as \eqref{mainJahnke}, and  the result holds.

The multinomial terms in \eqref{mainJahnke} represent the degradation of the molecules present at time $0$, and the Poisson term represents the number of the molecules born after time $0$ which survived up to time $t$.

We next generalise the result to arbitrary  $(m_1,\ldots,m_{d})\in\Z_{\geq0}^{d}$. First, we note that the multinomial part in the convolution \eqref{multinom1} is same as in \eqref{mainJahnke}, representing the distribution of the number of molecules at time $t$ that survived from the initial  $Z(0)$ molecules at time $0$. 

Let $V_{t}$ be the vector counting by species type the number of molecules  that are born in $(0,t]$ and survived until time $t$. The random variable  $V_t$ is a sum $\sum_{j=1}^dV^j_{t}$, where $V^j_{t}\in\Z_{\geq0}^{d}$ counts  the molecules that are produced by the reaction $0\to m_jS_j$ (if present) and subsequently transformed by conversion into other molecules and/or degraded.
Observe that the random variables $V_{t}^{j}$, $1\le j\le d$, are conditionally independent given $\Ft$, since they  are generated  by independent processes of production, and each molecule degrades or is transformed independently of the others, given $\Ft$ (Remark \ref{rem:independence}).  Further, observe that
\begin{equation}
\mathcal{L}\big(Z(t)\big| Z(u)=e_j, \Ft\big) \mathop{=}\mathcal{L}\big(\textup{Multi}\big(1, \Phi(u,t)e_j\big)\big| \Ft\big),\label{mul1}
\end{equation}
which follows from the first part of the proposition and  \eqref{funda_shift}.

Using similar arguments as those applied to show \eqref{multi}, we find that
\begin{equation}
\mathcal{L}\big(V_{t}^j\big|\Ft\big) =  \Conv_{\nu\in \Theta_{m_j,d}} \mathcal{L}\left(\left.\nu\,\textup{Pois}\left(\int_{0}^{t}\widetilde\lambda_{0j}(X(u))\,g^X_{u,t}(\nu, m_j)\,du\right)\right|\Ft\right),  \label{Wk}
\end{equation}
where  $g^X_{u,t}(\nu, m_j)$ is the probability that  $\textup{Multi}\!\left(m_j,\Phi(u,t)e_j\right)$ in \eqref{mul1} takes the value $\nu$. That is, each Poisson variable counts how often the configuration $\nu$ appears at time $t$.
Using the independence of the variables $V^{j}_{t}$, the Poisson part in the convolution \eqref{multinom1} follows from  \eqref{Wk}, and the proposition is proved.
\end{proof}

The probability distribution  $g^X_{u,t}(\nu, m_{j})$ (as a function of $t$) in Proposition \ref{P2}  takes  a similar form to that in  \eqref{eq:P1general}. The  difference being that Proposition \ref{P2}  describes a  multitype process in contrast to the univariate process in \eqref{eq:P1general} and that multiple molecules might be born at the same time. 

\section{The stationary distribution and ergodicity}
\label{StatDist}

In this section we study the long term behaviour of the models fulfilling  Assumption \ref{as0}. To motivate this further consider a process $\prX$ with 2 states $\{0,1\}$ and the following simple reaction system:
$$0\ce{<=>[\lambda_1][\lambda_2]}S,\qquad \lambda_1(x,z)=1,\quad \lambda_2(x,z)=xz.$$
The species $S$ is constantly produced while degradation at time $t$ depends on whether $X(t)$ is $1$ or $0$. Intuitively, if the process $\prX$ takes the value zero for a minuscule fraction of time one might expect a transient behaviour of $\prZ$. However, this is not the case.  The model falls in the category of  Section \ref{Ex} with $m=1$, $\kappa_1(x)=1$, and $\kappa_2(x)=x$. If $Z(0)=0$ then the distribution of $Z(t)$ is
$$\mathcal{L}\big(Z(t)\big| \Ft\big)=\mathcal{L}\left(\textup{Pois}\left(\int_0^t e^{-\int_u^t X(s)ds} du\right)\Big| \Ft\right).$$
If $\prX$ is ergodic, then   the intensity of the Poisson distribution converges in distribution as $t\to\infty$. Consequently, it follows using Levy's continuity theorem \cite{durrett}, that $Z(t)$ also converges in distribution and $Z(t)$ cannot be transient.
It is also a consequence of Theorem \ref{thm:ergodicity} below.

\subsection{Structural conditions for ergodicity}

To rigorously analyse this stability phenomenon we give graphical conditions under which the models satisfying Assumption \ref{as0} are ergodic for any initial condition. 

\begin{Definition}
Suppose Assumption \ref{as0} holds for a stochastic mass-action system with stochastic environment $(\clg, \Lambda, \prX)$, and let $S_i, S_j\in\cls$.  We say that 
\begin{itemize}
 \item $S_j$ is \emph{obtainable} from $S_i$, denoted $S_i\transf S_j$, if $S_i=S_j$, or there exists a sequence of reactions $S_{i_k}\to S_{i_{k+1}}$, $1\leq k\leq n-1$, with  $i_1=i$ and $i_n=j$,
 \item $S_i$ is \emph{properly produced}, denoted $0\transf S_i$, if there exists $0\to m_{i_1}S_{i_1}$ and a sequence of reactions $S_{i_k}\to S_{i_{k+1}}$, $1\leq k\leq n-1$, with $i_n=i$,
 \item $S_i$ is \emph{properly degraded}, denoted $S_i\transf 0$, if for every species $S_j$ obtainable from $S_i$, there is a  sequence of reactions  $S_{i_k}\to S_{i_{k+1}}$, $1\leq k\leq n-1$, with $i_1=j$, and such that $S_{i_n}\to 0$.
\end{itemize}
Furthermore, we write $S_i\ntransf S_j$, $0\ntransf S_i$, and $S_i\ntransf 0$ if the relation does not hold. 
\end{Definition}

\begin{Assumption}\label{asergodic*}
The process $\prX$ is ergodic with stationary distribution $\pi$. Furthermore, assume 
\begin{enumerate}[(a)]
 \item
 every properly produced species is properly degraded, 
  \item\label{item:finite_expectation} the expectation of $\widetilde\lambda_{0j}(\cdot)$ with respect to $\pi$ is finite for any $1\leq j\leq d$, that is
 \begin{equation*}
  \sum_{x\in\Gamma}\widetilde\lambda_{0j}(x)\pi(x)=\sum_{x\in\Gamma}\|B(x)\|_1\pi(x)<+\infty.
   \end{equation*}
\end{enumerate}
\end{Assumption}
 
The process $\prXZ$ can only be  ergodic for any initial condition if the properly produced species are also properly degraded. Hence, the essential restriction is \ref{item:finite_expectation}. 

Under Assumption \ref{as0} and \ref{asergodic*}, the relation  $\transf$ induces a partition of the species set 
$\cls=\cls_1\cup \cls_2\cup\dots\cup \cls_h \cup \cls_T\cup\cls_P$, $h\in\Z_{\ge 0}$, where each $\cls_i$ is a closed strongly connected component of the graph  associated to $\transf$, the set $\cls_P$ consists of the properly produced species (hence also properly degraded), and $\cls_T$ consists of the remaining (transient) species. The transient species are either properly degraded (but not produced) or converted into other species in one or more connected components, potentially both.

As an example, consider the reaction network $0\to 2S_1$, $S_1\to S_2\to 0$, $S_3\to S_4\to S_5$, $S_5\to S_4$. The partition induced by $\transf$ is $\cls=\{S_4,S_5\}\cup\{S_1,S_2\}\cup\{S_3\}$ with $\cls_P=\{S_1,S_2\}$ and $\cls_T=\{S_3\}$.

The first main result is the following.
\begin{Theorem}\label{thm:ergodicity}
Suppose Assumption \ref{as0} and \ref{asergodic*} are fulfilled. 
 Consider a stochastic  reaction network with stochastic environment, then  $\prXZ$ is ergodic for any initial condition.
  \end{Theorem}

By `ergodic for any initial condition' we understand the existence of a unique irreducible component of $\Gamma\times\Z_{\ge 0}^d$, that is reach with probability one, for any initial condition.  The process restricted to the irreducible component is positive recurrent (or ergodic, as component is  irreducible). Since $\Gamma$ by assumption is irreducible, the possible division of $\Gamma\times\Z_{\ge 0}^d$ is induced by the structure of the reaction network.

We prove the theorem by making use of properties of the fundamental matrix solution $\Phi(\cdot)$. The proof is in Section \ref{sec:ergodicity} and draws on material in Section \ref{sec:assumps}.

\subsection{Connecting  Assumption \ref{asergodic*} with properties of $\Phi(\cdot)$}
\label{sec:assumps}

Here we highlight some connections between properties of the matrix $\Phi(\cdot)$ and Assumption~\ref{as0} and \ref{asergodic*}. 

If the state space $\Gamma$ consists of just one element then the matrices $A_X$ and $B_X$ are constant, deterministic matrices, and the environment is not stochastic. Moreover, if $m_j\in\{0,1\}$ for all $1\leq j\leq d$, then the process $\prZ$ is a CTMC associated with a mono-molecular stochastic reaction network  \cite{jahnke2007solving}. Questions about ergodicity  and the form of the stationary distribution have been fully answered in this case \cite{jahnke2007solving}.  

Therefore, assume $\Gamma$ has at least two elements and let $x\in\Gamma$. Define $\tau^x_{0}\geq 0$ to be the first time $\prX$ hits the state $x$, and let $T^x_{0}$ denote the time spent in $x$. If $X(0)=x$, then $\tau^x_{0}=0$. Recursively define for $k\ge 1$,
\begin{equation*}
\tau^x_{k}=\inf\{t>\tau^x_{k-1}+T^x_{k-1}\colon X(t)=x\},\s\text{ and }\s T^x_{k}=\inf\{t>\tau^x_{k}\colon X(t)\neq x\}-\tau^x_{k}.
\end{equation*}
The time between the $(k-1)$-th and $k$-th visits to  $x$ is $\tau^x_k-\tau^x_{k-1}$ for $k\ge 1$. By the Markov property the random variables $\tau^x_k-\tau^x_{k-1}$, $k\ge 1$, are independent and identically distributed. Moreover, $\tau^x_0$ is independent of $\tau^x_k-\tau^x_{k-1}$ for all $k\ge 1$. If $\prX$ is ergodic then the waiting times $\tau^x_k$, $k\ge 1$, have positive finite expectation (recall,  $\Gamma$ is irreducible). If $X(0)\neq x$, then the same holds for $\tau^x_0$; otherwise $\tau^x_0=0$.

We state here a technical implication of Assumption~\ref{as0} and \ref{asergodic*} in terms of $\Phi(\cdot)$.

\begin{Lemma}\label{lem:norm_1_implications}
 Suppose Assumptions \ref{as0} and \ref{asergodic*} hold and that $\Gamma$ contains at least two elements. Then, there exists $\alpha\in\Z_{>0}$ such that for all $x\in\Gamma$,
 \begin{enumerate}[(a)]
\item\label{item:norm<1} $\displaystyle E\Big(\max_{\substack{1\leq i\leq d\\ S_i\transf 0}}\left\|\Phi(\tau^x_0,\tau^x_{\alpha})e_i\right\|_{1}\Big)<1$,
\item\label{item:irreducibility}  $\displaystyle E\Big(\min_{\substack{1\leq i,j\leq d\\S_i\transf S_j}}\Phi_{ij}(\tau^x_0,\tau^x_{\alpha})\Big)>0$,
\item\label{item:production_bounded} $\displaystyle E\left(\int_{\tau^x_0}^{\tau^x_{\alpha}}\left\| \Phi(u,\tau^x_{\alpha})B_{X}(u)\right\|_{1} \,du\right)<+\infty$.
\end{enumerate}
\end{Lemma}
\begin{proof} \ref{item:norm<1}
 For simplicity, let $S_{i_1}, S_{i_2}, \dots, S_{i_{\tilde{d}}}$, $\tilde d\ge 0$, be the properly degraded species. By Lemma~\ref{lem:interpretation}\ref{item:degradation_prob} we have $\|\Phi(\tau^x_{k},\tau^x_{k+k'})e_i\|_{1}\le 1$ for any $k,k'\in\Z_{\geq0}$, $1\leq i\leq d$, and $x\in\Gamma$. Hence it suffices to show that there exists $\alpha\in\Z_{>0}$ such that for all $x\in\Gamma$,
 \begin{equation}\label{eq:pos_probab_<1}
 P\Big(\max_{1\leq \ell\leq \tilde{d}}\|\Phi(\tau^x_0,\tau^x_{\alpha})e_{i_\ell}\|_1<1\Big)>0,
\end{equation}
For any $1\leq \ell\leq \tilde{d}$, there exists a directed path of reactions $S_{i_\ell}=S_{j_0}\to S_{j_1}\to \dots \to S_{j_{\alpha_{\ell}}} \to0$ $(=S_{j_{\alpha_\ell}+1}$ for convenience)  
with $\alpha_{\ell}\in\Z_{\geq0}$. For any $0\leq k\leq \alpha_{\ell}$, let $x_k\in\Gamma$ be such that $\lambda_{j_k,j_{k+1}}(x_k)>0$, and $x\in \Gamma$. Since $\Gamma$ is irreducible, with positive probability the process $\prX$ visits the states $x_0,x_1,\ldots,x_{\alpha_{\ell}}$ in that order, between time $\tau^x_0$ and $\tau^x_{\alpha_{\ell}}$ (if there is more than one visit to $x$  between visits to $x_{k-1}$ and $x_k$, then the path could be contracted). 
Hence, for any $1\leq \ell\leq \tilde{d}$ and any $x\in\Gamma$, there is  positive probability that a molecule of species $S_{i_\ell}$, present at time $\tau_0^x$, is degraded by time $\tau_{\alpha_\ell}^x$. Therefore, it follows from Lemma~\ref{lem:interpretation}\ref{item:probability} that
$$P\left(\|\Phi(\tau^x_0,\tau^x_{\alpha})e_{i_\ell}\|_{1}<1\right)>0\quad\text{for all }\quad\alpha\ge \alpha_{\ell}\text{ and all } x\in\Gamma.$$
To prove \eqref{eq:pos_probab_<1}, we  need to show that molecules of different properly degraded species can \emph{all} be degraded by the same return time to $x$ with positive probability.  Define
$$\widetilde{\alpha}=\sum_{1\leq \ell\leq \tilde{d}} \alpha_{\ell}.$$
Assume that a molecule of every properly degraded species is present at time $\tau_0^x$. With positive probability, the molecule of species $S_{i_1}$ is degraded by time $\tau^x_{\alpha_{1}}$, while the other molecules do not change. Then, with positive probability the molecule of species $S_{i_2}$ is degraded by time $\tau^x_{\alpha_{1}+\alpha_{2}}$, while the molecules of the other species do not change, and so on. In conclusion, with positive probability all the molecules of properly degraded species are degraded by time $\tau^x_{\widetilde{\alpha}}$, and \eqref{eq:pos_probab_<1} holds for any $\alpha\geq \widetilde{\alpha}$.

\ref{item:irreducibility} Assume $S_i\transf S_j$. By a similar argument as before, there exists $\alpha_{ij}\in\Z_{>0}$ such that a molecule of species $S_i$, present at time $\tau_0^x$, is transformed into a molecule of species $S_j$ by time $\tau_{\alpha_{ij}}^x$ with positive probability $\Phi_{ij}(\tau^x_0,\tau^x_{\alpha_{ij}})$, for any $x\in\Gamma$. Similarly as before, by choosing
$$\widehat\alpha=\sum_{\substack{1\leq i,j\leq d \\ S_i\transf S_j}}\alpha_{ij}$$
then the desired holds for  all $\alpha\geq\widehat\alpha$. Hence, \ref{item:norm<1} and \ref{item:irreducibility} hold at the same time for $\alpha\geq \max(\widetilde{\alpha}, \widehat{\alpha})$.

\ref{item:production_bounded} It follows from Lemma~\ref{lem:interpretation}\ref{item:norm_1_leq_1} that for any $x\in\Gamma$
 $$E\left(\int_{\tau^x_0}^{\tau^x_{\alpha}}\big\|\Phi(u,\tau^x_{\alpha})B_{X}(u)\big\|_{1} \,du\right)\leq E\left(\int_{\tau^x_0}^{\tau^x_{\alpha}}\big\|B_{X}(u)\big\|_{1} \,du\right)= \alpha\, E\left(\int_{\tau^x_0}^{\tau^x_1}\big\|B_{X}(u)\big\|_{1} \,du\right).$$
Hence, it suffices to show that the expectation of the right-hand side is bounded. By the definition of $B_X(\cdot)$, it suffices to show that
 $$E\left(\int_{\tau^x_0}^{\tau^x_1}\widetilde{\lambda}_{0i}(X(u)) \,du\right)<\infty$$
 for all $i$ such that $m_i\neq 0$. By ergodicity of $\prX$ and non-negativity of $\widetilde{\lambda}_{0i}(\cdot)$,
 $$ \lim_{t\to\infty}\frac{1}{t}\int_{0}^t\widetilde{\lambda}_{0i}(X(u))\,du=\sum_{x\in\Gamma}\widetilde{\lambda}_{0i}(x)\pi(x)\qquad\text{a.s.,}$$
 where $\pi$ denotes the stationary distribution of $\prX$. Since $\Gamma$ is irreducible and contains at least  two elements, $\lim_{n\to\infty} \tau^x_n=\infty$ a.s. It follows that
 $$\lim_{n\to\infty}\frac{1}{\tau^x_n} \int_0^{\tau^x_n}\widetilde{\lambda}_{0i}(X(u))\,du=\sum_{x\in\Gamma}\widetilde{\lambda}_{0i}(x)\pi(x)\qquad\text{a.s.}$$
Since $\tau^x_0<\infty$ a.s., then  by  the strong law of large numbers,   it holds with probability one that
   \begin{align*}
 \lim_{n\to\infty}\frac{1}{\tau^x_n} \int_0^{\tau^x_n}\widetilde{\lambda}_{0i}(X(u))\,du & =
 \lim_{n\to\infty}\frac{1}{\tau^x_n-\tau^x_0} \int_{\tau^x_0}^{\tau^x_n}\widetilde{\lambda}_{0i}(X(u))\,du\\
 &=\lim_{n\to\infty}\frac{n}{\sum_{k=1}^n(\tau^x_k-\tau^x_{k-1})} \frac{1}{n}\sum_{k=1}^n\int_{\tau^x_{k-1}}^{\tau^x_k}\widetilde{\lambda}_{0i}(X(u))\,du \\ 
 &=\frac{1}{E(\tau^x_1-\tau^x_0)}E\left(\int_{\tau^x_0}^{\tau^x_1}\widetilde{\lambda}_{0i}(X(u))\,du\right).
  \end{align*}
Consequently
 $$E\left(\int_{\tau^x_0}^{\tau^x_1} \widetilde{\lambda}_{0i}(X(u))\,du\right)=E(\tau^x_1-\tau^x_0) \sum_{x\in\Gamma}\widetilde{\lambda}_{0i}(x)\pi(x).$$
 The proof is concluded by the fact that the right-hand side is bounded by Assumption \ref{asergodic*}. 
\end{proof}

To gain intuition about the role played by $\alpha$ in Lemma~\ref{lem:norm_1_implications}, we give an example of a mass-action system with stochastic environment and find an integer $\alpha$ satisfying Lemma~\ref{lem:norm_1_implications}(a).

\begin{Example}
Let $\prX$ be a Markov process with two states $0$ and $1$ such that  $q_{10}$ is the transition rate from 1 to 0, and $q_{01}$ that from 0 to 1. Consider the following stochastic mass-action system with stochastic environment:
$$0\ce{->[\lambda_1]} S_1\ce{->[\lambda_2]} S_2\ce{->[\lambda_3]} S_3\ce{->[\lambda_4]} 0$$
with transition rates
$$\begin{array}{ll}
 \lambda_1(x,z)=\kappa_1 & \lambda_2(x,z)=\kappa_2xz_1 \\
 \lambda_3(x,z)=\kappa_3(1-x)z_2 & \lambda_4(x,z)=\kappa_4xz_3 \\
\end{array}
$$
for positive constants $\kappa_1$, $\kappa_2$, $\kappa_3$, and $\kappa_4$. Assumption \ref{as0} and \ref{asergodic*} hold, and all species are properly produced and properly degraded. Hence Lemma~\ref{lem:norm_1_implications} applies. The conclusions of the lemma hold for $\alpha=2$, but not for $\alpha=1$. It follows from the simple observation that a molecule of species $S_1$, present at time 0, can only be  removed from the system  after the process $\prX$ has visited the state 1 twice: $S_1$ can only be transformed into $S_2$ when $X(t)=1$, $S_2$ can only be transformed into $S_3$ when $X(t)=0$, and finally $S_3$ can be degraded only when $X(t)=1$. Using  Lemma~\ref{lem:interpretation}, it follows that
$$\left\|\Phi(\tau^1_0,\tau^1_1)e_i\right\|_1=\sum_{j=1}^3 P_X^{(i,j)}(\tau^1_0,\tau^1_1)=1\quad\text{a.s.,}\quad
\quad\left\|\Phi(\tau^1_0,\tau^1_2)e_i\right\|_1=\sum_{j=1}^3 P_X^{(i,j)}(\tau^1_0,\tau^1_2)<1\quad\text{a.s.}$$
\end{Example}

\subsection{A recurrence relation}

A random variable $V$ with values in $\R^n$ satisfies a Stochastic Recurrence Equation (SRE) characterised by a pair of random variables  $(Q_1,Q_2)$ with values in $\R^{m\times n}\times\R^m$ if  the following holds
\begin{equation}
V\,\,\sim\,\, Q_{1}V+Q_{2},\qquad V\,\,\bigCI \,\,(Q_{1},Q_{2}).\label{defstochrec}
\end{equation}
Furthermore, if any two random variables  satisfying \eqref{defstochrec} are identical in law, then this law is said to be the weakly unique solution of \eqref{defstochrec}. This type of equation appears in modelling perpetuities which are  applicable in financial mathematics \cite{buraczewski2016stochastic}. General existence and uniqueness conditions for the law of $V$ are given in \cite{buraczewski2016stochastic}, where other applications are also mentioned.

Define
\begin{equation}\label{eq:hattau}
n_t^x=\sup\{n\colon\tau^x_{n}\le t\},\quad \text{and}\quad \tau^x(t)=\tau^x_{n_t^x},
\end{equation}
where $n_t^x$ is the number of times  $\prX$ visits  $x\in\Gamma$ before time $t$.
For  $k\in\Z_{\geq0}$, define
\begin{align}\label{eq:def_C}
C^x_k&=\Phi(\tau^x_k,\tau^x_{k+1}),\quad\text{and}\quad
D^x_k=\int_{\tau^x_k}^{\tau^x_{k+1} }\Phi(u,\tau^x_{k+1} )B_{X}(u)\,du.
\end{align}
By the strong Markov property, $(C^x_k,D^x_k)_{k\ge 0}$, is a sequence of independent, identically distributed random variables. For notational convenience, let
$$C^x_{-1}=\Phi(0,\tau^x_0),\quad\text{and}\quad D^x_{-1}=\int_{0}^{\tau^x_0 }\Phi(u,\tau^x_0 )B_{X}(u)\,du.$$
Then, using \eqref{eq:phi_tu}, we obtain
\begin{align}\label{eq:rel_C_Phi}
\Phi(t)&=\Phi(\tau^x(t),t) \prod_{k=-1}^{n_t^x-1}C^x_k, \\
\label{eq:rel_D_W}
W(t)&=\int_{\tau^x(t)}^t \Phi(u,t)B_X(u)\,du\,+ \Phi(\tau^x(t),t)\sum_{k=-1}^{n_t^x-1} \left(\prod_{i=k+1}^{n_t^x-1} C^x_i\right)D^x_k.
\end{align}

The proof of the next two lemmas can be found in the Appendix.  In the lemma below the order of the matrices in the product of $\Phi^x_\infty$ and $W^x_\infty$ is reversed compared to \eqref{eq:rel_C_Phi} and \eqref{eq:rel_D_W}. In general we cannot assure the existence of the random variables $\prod_{k=0}^{\infty}C^x_k$ and $\sum_{k=0}^{\infty} (\prod_{i=0}^{k-1} C^x_i)D^x_k$, intended as strong convergence limits. More details can be found in the proof of Lemma~\ref{lem:R}, and from \cite{infinite_matrix_product}.

\begin{Lemma}\label{lem:R}
 Suppose Assumption~\ref{as0} and \ref{asergodic*} hold, and $\cls_T=\emptyset$. Furthermore, assume $\Gamma$ is not a singleton set, and $X(0)=x\in\Gamma$. Then, 
 \begin{equation}\label{eq:def_Phi_W_infty}
 \Phi^x_\infty=\prod_{k=\infty}^{0}C^x_k\qquad\text{and}\qquad W^x_\infty=\sum_{k=0}^{\infty} \Big(\prod_{i=k-1}^{0} C^x_i\Big)D^x_k
\end{equation}
 exist a.s. Moreover,
 \begin{equation}\label{eq:goal}
  \big(\Phi(\tau^x(t)),W(\tau^x(t))\big)\todist{}(\Phi^x_\infty,W^x_\infty)\quad \text{for}\quad t\to\infty,
 \end{equation}
$ E(\|W^x_\infty\|_1)<\infty,$
 and for all continuity sets  $A$ of $(\Phi^x_\infty,W^x_\infty)$,  
 \begin{equation}\label{eq:conditional_limit}
  \lim_{t\to\infty}P\Big(\big(\Phi(\tau^x(t)), W(\tau^x(t))\big)\in A\,|\, X(t)=x\Big)=P\big((\Phi^x_\infty,W^x_\infty)\in A\big).
 \end{equation}
 The columns of the matrix $\Phi^x_\infty$ corresponding to the species in $\cls_i$, $i=1,\ldots,h$, are identical.
\end{Lemma}
 
 As the transient species are never produced but are all eventually degraded or transformed into other species, the assumption $\cls_T=\emptyset$ is more of convenience than necessity.
  
\begin{Lemma}\label{lem:uniqueness} Suppose Assumption~\ref{as0} and \ref{asergodic*} hold and that   $\Gamma$ contains at least two elements. 
Then  $(\Phi^x_\infty,W^x_\infty)$ is a solution to the  SRE,  
\begin{equation}\label{matrixrec}
\begin{pmatrix} V_1 \\ V_2\end{pmatrix} \,\,\sim \,\,\begin{pmatrix} C & 0 \\ 0 & C\end{pmatrix} \begin{pmatrix} V_1 \\ V_2\end{pmatrix} +\begin{pmatrix} 0 \\ D\end{pmatrix},\qquad (V_1,V_2)\,\bigCI \,(C, D),
\end{equation}
where $(C,D)$ is distributed as $(C^x_0,D^x_0)$. Moreover, $W^x_\infty$ is the weakly unique solution to the  bottom line  $V_2\sim CV_2+D$ in \eqref{matrixrec}.
\end{Lemma}

\subsection{Proof of Theorem \ref{thm:ergodicity}} 
  \label{sec:ergodicity}
  
Before proving Theorem~\ref{thm:ergodicity} we need some preliminary results. The following lemma relies on the characterisation of the distribution of $Z(t)$ for $t\ge 0$ given in Proposition~\ref{P2}. Recall that the state space of $\prX$ is irreducible, hence whether the process $\prW$ is tight or not holds regardless of the initial condition $X(0)$.

\begin{Lemma}\label{cor:ergodic}
Suppose Assumption \ref{as0} holds. Furthermore, assume the process $\prX$ is ergodic and let $\prW$ be as in \eqref{eq:convergence}. Then, $\prW$ is tight if and only if the process $\prXZ$ is ergodic for any initial condition.
\end{Lemma}

The proof of the Lemma is in the Appendix. 

\begin{Lemma}\label{int-tight}
Assume that $\prX$ is ergodic and that $\Gamma$ contains at least two elements. Let $x\in\Gamma$ and define $\tau^x(t)$ as in \eqref{eq:hattau}. Then, the process 
$$\{W(t)-W(\tau^x(t))\colon t\ge 0\}$$
is tight.
\end{Lemma}
\begin{proof}
 The process $\prX$ is ergodic, hence tight. Moreover, the state space $\Gamma$ is by assumption discrete. It follows that for any $\varepsilon>0$ there exists a finite set $\Upsilon_\varepsilon\subseteq\Gamma$ such that
 $$\inf_{t\ge 0}P(X(t)\in \Upsilon_\varepsilon)>1-\frac{\varepsilon}{2}.$$
 By assumption, the process $\prX$ is non-explosive. Hence, since $\Upsilon_\varepsilon$ is finite, it follows that for any $T>0$ (with the given $\varepsilon$) there exists a finite set $\widetilde{\Upsilon}_{\varepsilon, T}\subseteq\Gamma$ such that
 $$\inf_{t\ge 0}P\Big(X(u)\in \widetilde{\Upsilon}_{\varepsilon, T}\text{ for all }u\in [t,t+T]\Big)>1-\frac{\varepsilon}{2}.$$
 The finiteness of $\widetilde{\Upsilon}_{\varepsilon, T}$ in turn implies that
 $$\inf_{t\ge 0}P\Big(\|B_X(u)\|_1\leq M_{\varepsilon, T}\text{ for all }u\in [t,t+T]\Big)>1-\varepsilon,$$
 where
 $$M_{\varepsilon, T}=\max_{\widetilde{x}\in \widetilde{\Upsilon}_{\varepsilon, T}}\|B(\widetilde{x})\|_1<\infty.$$
 By standard renewal theory  the age process $\{t-\tau^x(t)\colon t\ge 0\}$ is tight \cite{durrett}. In particular, a unique stationary distribution exists with density with respect to the Lebesgue measure  \cite{durrett},
  $$f^x(u)=\frac{P(\tau^x_1\geq u)}{E(\tau^x_1)}\quad\text{ for all }\quad u\geq0.$$
 Hence, for all $\varepsilon>0$ there exists $M'_\varepsilon$ such that
 $$\inf_{t\ge 0}P\Big(t-\tau^x(t)\leq M'_\varepsilon\Big)>1-\frac{\varepsilon}{2}.$$
 By Lemma~\ref{lem:interpretation},
 \begin{align*}
  \|W(t)-W(\tau^x(t))\|_1&=\left\|\int_{\tau^x(t)}^{t}\Phi(u,t)B_{X}(u)\,du\right\|_1\\
  &\leq \int_{\tau^x(t)}^{t}\|B_{X}(u)\|_1\,du\leq (t-\tau^x(t))\cdot\sup_{u\in[\tau^x(t), t]}\|B_{X}(u)\|_1
 \end{align*}
 It follows that for any $\varepsilon>0$,
 $$\inf_{t\in\R_{\geq0}}P\Big(\|W(t)-W(\tau^x(t))\|_1\leq M'_\varepsilon \cdot M_{\varepsilon, M'_\varepsilon}\Big)>1-\varepsilon,$$
 which concludes the proof.
\end{proof}

We are now ready to prove Theorem~\ref{thm:ergodicity}.

\begin{proof}[Proof of Theorem~\ref{thm:ergodicity}]
First, assume that $\Gamma$ has at least two elements. Due to Lemma~\ref{cor:ergodic}, it is sufficient to prove that the process $\prW$ is tight for any initial condition $X(0)$. Assume $X(0)=x\in\Gamma$. Since
\begin{equation*}
W(t)= \big(W(t)-W(\tau^x(t))\big)+W(\tau^x(t))\quad\text{for all}\quad t\ge 0,
\end{equation*}
the process $\prW$ is tight because of Lemmas~\ref{lem:R} and \ref{int-tight}. The proof is then concluded for $\Gamma$ containing at least two elements. 

If $\Gamma$ contains a single element $x$, then we  consider a modification of the model with $\Gamma'=\{x_1, x_2\}$ and $\prXp$ being any irreducible CTMC on $\Gamma'$. Since $\Gamma'$ is finite, $\prXp$ is ergodic \cite{norris1998markov}. We define a stochastic reaction system with stochastic environment $(\clg, \Lambda', \prXp)$ by letting 
$$\lambda_r(x_1,z)=\lambda_r(x_2,z)=\lambda_r(x,z)\quad\text{for all}\quad z\in\Z_{\geq0}^d,$$
and all reactions $y_r\to y'_r$ of $\clg$.
Then the distribution of $\prZp$ is the same as that of $\prZ$. Specifically, $\prXZ$ is ergodic if and only if $\prXZp$ is ergodic, and since $\Gamma'$ contains at least two elements the proof is concluded.
\end{proof}

\subsection{Characterisation of the stationary distribution}

We characterise the explicit structure of the stationary distribution of models satisfying Assumption \ref{as0} and \ref{asergodic*} in  the case 
$m_i\in\{0,1\}$ for  $1\leq i\leq d$. If $d=1$, the model is a $M/M/\infty$ queue with stochastic environment, and the stationary distribution has been characterised by a recursive equation  under the assumption that $\Gamma$ has finitely many states  \cite{o1986m}.

For  $x\in\Gamma$, let $q_x$ be the rate of the exponential holding time of $\prX$ in $x$, and let $G_x\colon \R_{\ge 0}\to \R^{d}$ be the function defined as
\begin{equation}\label{eq:G}
    G_x(s)= \int_0^{s} e^{A(x)u}B(x)\,du=\sum_{n= 0}^{\infty}\frac{1}{(n+1)!}A^n(x)B(x)s^{n+1},
\end{equation} 
where $A(x)$ and $B(x)$ are as given in \eqref{eq:AX} and \eqref{eq:BX}, respectively.

\begin{Theorem}\label{T1}
Suppose Assumption \ref{as0} and \ref{asergodic*} hold, and that $\cls_T=\emptyset$. Further, assume that $m_i\in\{0,1\}$ for $1\leq i\leq d$. Then, $\prXZ$ is ergodic, and the irreducible closed sets are
\begin{equation*}
\Xi_{n_1, \dots, n_h}=\Big\{(x,z)\in\Gamma\times\Z^d_{\geq0}\colon \sum_{S_j\in\cls_i}z_j=n_i\text{ for  }i=1,\ldots,h\Big\},
\end{equation*}
with $n_1, \dots, n_h\in\Z_{\geq0}$. Moreover, for any $(x,z)\in\Xi_{n_1, \dots, n_h}$
\begin{equation}
\widehat\pi(x,z) = \pi(x)\int_{\R_{\geq0}^{d\times d}\times\R_{\geq0}^d} P\Big(N(w)+\sum_{i=1}^h M_i(n_i,ue_{j(i)})=z\Big)  \mu_x(du,dw),\label{invariant}
\end{equation}
where $N(w)\sim\textup{Pois}(w)$ and $M_i(k,v)\sim\textup{Multi}(k,v)$, $j$ is any function $j\colon\{1,\ldots,h\}\to\{1,\dots,d\}$ such that $S_{j(i)}\in\cls_i$ for all $i$, $\pi$ is the stationary distribution of $\prX$ and $\mu_x$ is the distribution of
\begin{equation}\label{mainres}
\big(e^{A(x)U_{x}}\Phi^x_\infty,e^{A(x)U_{x}}W^x_\infty+G^{x}(U_{x})\big),
\end{equation}
with $U_x\sim \textup{Exp}(q_x)$ being independent of $(\Phi^x_\infty,W^x_\infty)$.
\end{Theorem}

\begin{proof} If $\Gamma$ consists of a single element then we proceed as in the proof of Theorem \ref{thm:ergodicity} to obtain an equivalent process on a state space with two states.  Hence we assume $\Gamma$ has at least two elements.

Ergodicity for any initial condition follows from Theorem~\ref{thm:ergodicity}. The form of the irreducible sets
follow from the definition of the partition of $\cls$.

Let $(x,z)\in \Xi_{n_1, \dots, n_h}$. To obtain the stationary distribution evaluated at $(x,z)$ we analyse 
\begin{multline}\label{eq:split}
P\big(X(t)=x,Z(t)=z\big|X(0)=x,Z(0)=z\big)\\
=P\big(X(t)=x\big| X(0)=x\big)P\big(Z(t)=z\big|Z(0)=z,X(0)=x,X(t)=x \big),
\end{multline}
which converges to $\widehat\pi(x,z)$ as $t\to\infty$ by ergodicity. The first term in the product 
 is a consequence of \eqref{parasite}. By ergodicity of $\prX$,  to show that \eqref{eq:split} converges to \eqref{invariant} as $t\to\infty$, it suffices to show that 
$$\lim_{t\to\infty}P(Z(t)=z\big|Z(0)=z,X(0)=x,X(t)=x)
=\int_{\R_{\geq0}^{d\times d}\times\R_{\geq0}^d}P\Big(N(w)+\sum_{i=1}^h M_i(n_i,ue_{j_i})=z\Big)  \mu_x(du,dw).$$
By Proposition~\ref{P2}, 
 \begin{equation}\label{eqgrgrh}
    \mathcal{L}\big(Z(t)\big|Z(0)=z,X(0)=x,X(t)=x \big)=\mathcal{L}\Big(N(W(t))+\sum_{i=1}^d \widetilde M_i(z_i,\Phi(t)e_i))\Big| X(0)=x,X(t)=x\Big),
     \end{equation}
where $N(w)\sim\textup{Pois}(w)$ and $\widetilde M_i(k,v)\sim\textup{Multi}(k,v)$, $i=1,\ldots,d$, are independent random variables. Furthermore, conditional on the event $\{X(0)=x,X(t)=x\}$,
\begin{equation}
(\Phi(t),W(t))\sim\Big(e^{A(x)(t-\tau^x(t))}\Phi(\tau^x(t))\,,\,G_x(t-\tau^x(t))+e^{A(x)(t-\tau^x(t))}W(\tau^x(t))\Big),\label{V_t}
\end{equation}
where $G_x$ is  defined in \eqref{eq:G}.  Indeed, note that
$$G_x(t-\tau^x(t))=\int_0^{t-\tau^x(t)} e^{A(x)u}B(x)\,du=\int_{\tau^x(t)}^t e^{A(x)(t-u)}B(x)\,du.$$
Motivated by \eqref{eqgrgrh} and \eqref{V_t}, we study the limit of the joint distribution 
\begin{equation*}
\mathcal{L}\big(t-\tau^x(t), \Phi(\tau^x(t)), W(\tau^x(t))\,\big|\, X(0)=x,X(t)=x\big)
\end{equation*}
as $t\to\infty$. For any $u, s\in\R_{\geq0}$ with $u\leq s$, define the event
$$B_{u,s}=\big\{\text{there is no jump of $\prX$ in } (s-u,s]\big\}.$$
Observe that for any continuity set $A\subseteq \R^{d\times d}\times\R^d$ of $(\Phi^x_\infty, W^x_\infty)$, and for any $s\le t$, the event
$$\big\{X(t)=x, t-\tau^x(t)>s, \big(\Phi(\tau^x(t)), W(\tau^x(t))\big)\in A, X(0)=x \big\}$$
coincides with the event
$$\big\{ B_{s,t}, X(t-s)=x, \big(\Phi(\tau^x(t-s)), W(\tau^x(t-s))\big)\in A, X(0)=x\big\},$$
since conditioned on $t-\tau^x(t)>s$, one has $\tau^x(t)=\tau^x(t-s)$. Using the above equality one gets
\begin{equation}\label{RHStoprove}
\begin{split}
 P\big(t-\tau^x(t)>s, \big(\Phi(\tau^x(t))&, W(\tau^x(t))\big)\in A\,\big\vert\,X(0)=x,X(t)=x\big)\\
  & =P\big(B_{s,t}\,\big\vert\,X(0)=x, X(t-s)=x,\big(\Phi(\tau^x(t-s)), W(\tau^x(t-s))\big)\in A\big) \\
  &\quad\cdot \,P\big(\big(\Phi(\tau^x(t-s)), W(\tau^x(t-s))\big)\in A\,\big|\,X(0)=x, X(t-s)=x\big) \\
  &\quad\cdot\,\frac{P(X(0)=x,X(t-s)=x)}{P(X(0)=x,X(t)=x)}. 
\end{split}
\end{equation}
Regarding the first term on the right-hand side of \eqref{RHStoprove}, observe that $(\Phi(\tau^x(t-s)),W(\tau^x(t-s)))$ is $\mathcal{F}_{t-s}$ measurable. Hence, by using Markov property of $\prX$, we obtain
\begin{align*}
 P\big(B_{s,t}\,\big\vert\,X(0)=x, X(t-s)=x,\big(\Phi(\tau^x(t-s)), W(\tau^x(t-s))\big)\in A\big)&\\
 &\hspace{-5cm}=P\big(B_{s,t}\,\big\vert\,X(0)=x, X(t-s)=x\big)=e^{q_x s}, 
\end{align*}
which leads to the  expression
\begin{align*}
P\big(t-\tau^x(t)>s, \big(\Phi(\tau^x(t)), W(\tau^x(t))\big)\in A\,\big|\,X(0)=x,X(t)=x\big)&\\
&\hspace{-245pt}= e^{q^xs}
 P\big(\big(\Phi(\tau^x(t-s)), W(\tau^x(t-s))\big)\in A\,\big\vert\,X(0)=x, X(t-s)=x\big)\frac{P(X(0)=x,X(t-s)=x)}{P(X(0)=x,X(t)=x)}.
\end{align*}
  Note that the third term on the right-hand side of the above equality converges to $1$ as $t\to \infty$. Regarding the second term, it follows from \eqref{eq:conditional_limit} in Lemma~\ref{lem:R} that 
\begin{equation*}
\lim_{t\to\infty}P\big(\big(\Phi(\tau^x(t-s)), W(\tau^x(t-s))\big)\in A\,\big\vert\,X(0)=x, X(t-s)=x\big)=P\big(\big(\Phi^x_\infty,W^x_{\infty}\big)\in A\big),
\end{equation*}
using that $A$ is a continuity set for $(\Phi^x_\infty, W^x_\infty)$.
Hence, $t-\tau^x(t)$ and $(\Phi(\tau^x(t)),W(\tau^x(t)))$ are asymptotically conditionally independent given $\{X(0)=x,X(t)=x\}$, as $t$ goes to infinity. Moreover, conditioned on $\{X(0)=x,\,\,X(t)=x\}$, $t-\tau^x(t)$ converges in distribution to an exponential random variable with rate $q_x$. Now using that $\Phi(t)$ converges to a matrix with identical columns for the species in $\cls_i$, $i=1,\ldots,h$ (Lemma \ref{lem:R}),   Levy's continuity theorem \cite{durrett}, \eqref{eqgrgrh}, and \eqref{V_t}, it follows that \eqref{invariant} holds, which conludes the proof.
\end{proof}

The condition $\cls_T=\emptyset$ is not strict. If some transient species are present at time zero then they will eventually be degraded or converted, but never produced. Hence at stationarity there are no transient species  present. Moreover, if  all species are properly produced and degraded, that is, if $\cls=\cls_P$, then the multivariate terms in \eqref{invariant} are absent. Furthermore,  it follows from Lemma \ref{lem:norm_1_implications}\ref{item:norm<1}  that $\|\Phi^x_\infty\|_1=0$, hence from Theorem \ref{T1},
\begin{equation}\label{eq:inv_pois}
\widehat\pi(x,z) = \pi(x)\int_{\R_{\geq0}^d} P(\textup{Pois}(w)=z)  \mu_x(dw).
\end{equation}
We use this equation to find an expression for the factorial moments,  in  particular, for the case $d=1$ and $m=1$. In this scenario, $\prZ$ is distributed as an M/M/$\infty$ queue with stochastic environment (with $\Gamma$ potentially infinite).

\begin{Corollary} 
Suppose Assumption \ref{as0} and \ref{asergodic*} hold. Furthermore, assume $\cls=\cls_P$ 
and that $m_i\in\{0,1\}$ for all $1\leq i\leq d$. Let $\widehat\pi$ be the stationary distribution of $\prXZ$ given in \eqref{eq:inv_pois} and let $(X,Z)$ be a random variable with distribution $\widehat\pi$. For $q\in\Z_{\geq0}^{d}$, denote the $q$-th factorial moment by
 $$m_q=E\left(\frac{Z!}{(Z-q)!}1_{\{z\in\Z_{\ge 0}^d\colon z\geq q\}}(Z)\right),$$
 where $z!=\prod_{i=1}^d z_i!$
 Then, 
\begin{equation}\label{mom1}
m_q= \displaystyle \sum_{x\in\Gamma}\pi(x) \int_{\R_{\geq0}^d} u^q\mu_{x}(du). 
\end{equation}
Moreover, if $d=1$, 
then $m_0=1$, and  $m_q$, $q\in\Z_{>0}$, can be found iteratively by
\begin{equation*}
 m_q=\frac{1}{1-E\Big((C^x_0)^q\Big)}\sum_{i=0}^{q-1} \binom{q}{i} E\Big((C^x_0)^{i}(\widetilde{D}^{x}_{0})^{q-i}\Big)m_i
\end{equation*}
where
$$\widetilde{D}^{x}_{0}=
\begin{cases}
 \dfrac{\widetilde\lambda_{01}(x)}{\widetilde\lambda_{10}(x)}(1-C^x_0)\Big(1-e^{-\widetilde\lambda_{10}(x)U_x}\Big)+e^{-\widetilde\lambda_{10}(x)U_x}D^x_0 & \text{if }\widetilde\lambda_{10}(x)>0,\\[1.5em]
 (1-C^x_0)U_x\widetilde\lambda_{01}(x)+e^{-\widetilde\lambda_{10}(x)U_x}D^x_0 & \text{if }\widetilde\lambda_{10}(x)=0,
\end{cases}
$$
and $U_x$ is an exponential random variable with rate $q_x$, independent of  $(C^x_0,D^x_0)$.
\end{Corollary}

\begin{proof} If $\Gamma$ consists of a single element then we proceed as in the proof of Theorem \ref{thm:ergodicity} to obtain an equivalent process on a state space with two states.  Hence we assume $\Gamma$ has at least two elements.

Since $E(\frac{Z'!}{(Z'-q)!})=u^q$ for $u\in\R_{\geq0}^n$, $q\in\Z_{\geq0}^n$, and $Z'\sim \textup{Pois}(u)$,
 it holds that
 \begin{align*}
  E\left(\frac{Z!}{(Z-q)!}1_{\{z\in\Z_{\ge 0}^d\colon z\geq q\}}(Z)\right)
  &=\sum_{x\in\Gamma}\pi(x)E\left(\left.\frac{Z!}{(Z-q)!}1_{\{z\in\Z_{\ge 0}^d\colon z\geq q\}}(Z)\right| X=x\right),\\
  &=\sum_{x\in\Gamma}\pi(x) \int_{\R_{\geq0}^d} u^q\mu_{x}(du),
 \end{align*}
 which proves \eqref{mom1}. Assume that $d=1$. 
By Theorem \ref{T1}, Lemma \ref{lem:uniqueness}  and the remarks above the corollary, $\mu_x$ is the weakly unique solution of the SRE \eqref{mainres}, 
 \begin{equation}\label{eq:from_here}
  V\,\,\sim\,\, C^x_0V+\big((1-C^x_0)G^x(U_x)+e^{-\widetilde\lambda_{10}(x)U_x}D^x_0\big), \qquad V \bigCI \,(C^x_0,D^x_0).
\end{equation}
 Note that
 $$G_x(U_x)=\int_0^{U_x}e^{-\widetilde\lambda_{10}(x)u}\widetilde\lambda_{01}(x)du=
 \begin{cases}
  \dfrac{\widetilde\lambda_{01}(x)}{\widetilde\lambda_{10}(x)}\Big(1-e^{-\widetilde\lambda_{10}(x)U_x}\Big) & \text{if }\widetilde\lambda_{10}(x)>0,\\[1.5em]
  U_x\widetilde\lambda_{01}(x)& \text{if }\widetilde\lambda_{10}(x)=0.
 \end{cases}$$
 Hence, \eqref{eq:from_here} becomes
  \begin{equation*}
  V\,\,\sim\,\, C^x_0V+\widetilde{D}^{x}_{0}, \qquad V \bigCI \,(C^x_0,\widetilde{D}^{x}_{0}).
 \end{equation*}
 It follows from the binomial theorem and the independence of $V$ and $(C^x_0,\widetilde{D}^{x}_{0})$ that for  $q\in\Z_{>0}$,
 \begin{equation}\label{kasdnflkjswDNFLKJS}
       E(V^q)=E\Big((C^x_0V+\widetilde{D}^{x}_{0})^q\Big)=E(V^q)E\Big((C^x_0)^{q}\Big)+\sum_{i=0}^{q-1} \binom{q}{i} E(V^i)E\Big((C^x_0)^{i}(\widetilde{D}^{x}_{0})^{q-i}\Big).
 \end{equation}

By Lemma~\ref{lem:norm_1_implications}\ref{item:norm<1} we have $E((C^x_0)^\alpha)<1$ for some positive $\alpha$, which implies $E((C^x_0)^{q})<1$ for $q\geq1$, as $d=1$ and $C^x_0\le 1$ a.s.  Hence by \eqref{kasdnflkjswDNFLKJS} we obtain
 \begin{align*}
     \int_0^\infty u^q\mu_{x}(du)=E(V^q)&=\frac{1}{1-E\Big((C^x_0)^{q}\Big)}\sum_{i=0}^{q-1} \binom{q}{i} E(V^i)E\Big((C^x_0)^{i}(\widetilde{D}^{x}_{0})^{q-i}\Big)\\
     &=\frac{1}{1-E\Big((C^x_0)^{q}\Big)}\sum_{i=0}^{q-1} \binom{q}{i} E\Big((C^x_0)^{i}(\widetilde{D}^{x}_{0})^{q-i}\Big)\int_0^\infty u^i\mu_{x}(du),
 \end{align*}
 and the proof is concluded by \eqref{mom1}.
\end{proof}

\section{Generating samples from the conditional stationary distribution}\label{sample}

The theory provides a way to simulate observations from the long-term dynamics of a model under the conditions of Theorem~\ref{T1}, given that the environmental state is known. The goal of this section is to show how to simulate from the conditional stationary distribution
$$\widehat\pi(\cdot\,|\, x)=\frac{\widehat\pi(x,\cdot)}{\pi(x)}.$$
This is often desirable in the setting of biochemistry, as typically only a portion of the reactants is observable. In our setting, the observed variables might  be considered the environment.

Let $(X,Z)$ denote a random variable with distribution $\widehat\pi$. If it is possible to simulate from the  distribution $\pi$ of $X$, then it would also be possible to simulate from the joint  distribution of $(X,Z)$  by first simulating $X$ according to $\pi$ and then $Z$ according to the conditional distribution. For simplicity, in this section we assume that $\cls=\cls_P$. Hence, according to Theorem~\ref{T1}, the stationary distribution is characterised by the distribution of $W_\infty^x$. With a slight abuse of notation, we denote by $\mu_x$ the distribution of $e^{A(x)U_{x}}W_\infty^x+G^{x}(U_{x})$, with $U_x$ as in Theorem~\ref{T1}. Then,
in order to simulate from $\widehat\pi(\cdot\,|\,x)$, by Theorem~\ref{T1} it suffices to generate a sample $U$ from $\mu_x$, at which point
$$\widehat\pi(\cdot\,|\, x) = \mathcal{L}(\textup{Pois}(U)).$$
Here we show how to generate a sample with approximate distribution $\mu_x$ under the additional simplifying assumption that Lemma~\ref{lem:norm_1_implications} holds with $\alpha=1$, using that $\mu_x$ is the weakly unique solution of \eqref{mainres}.

First, for some $n\in\mathbb{N}$ we generate a sequence $(C_{0,i}^x, D_{0,i}^x)_{i=1,\ldots,n}$ of $n$ independent realisations of $(C_0^x, D_0^x)$. To this aim, it is sufficient to generate $n$ independent samples of the process $\prX$ with $X(0)=x$, until the stopping time $\tau^x_1$.
Once $K$ independent samples of $\prX$ are available, the desired sequence $(C_{0,i}^x, D_{0,i}^x)_{i=1,\ldots,n}$ can be  obtained by means of  \eqref{fundjump} and \eqref{phiprod}. Generate an observation  $U_x$ from an  exponential random variable with rate $q_x$, independent of $(C_{0,i}^x, D_{0,i}^x)_{i=1,\ldots,n}$. Let $V_0=0$ and define recursively for  $1\leq i\leq n$,
$$V_i=C_{0,i}^x V_{i-1}+D_{0,i}^x.$$
We  take $V^{*}_{n}=e^{A(x)U_x}V_n+G^{x}(U_x) $ as an approximate sample from $\mu_x$. The following result gives an estimate of the error made with this approximation in terms of the Wasserstein metric on $(\Z_{\geq0}^d, \|\cdot\|_1)$,
$$\mathcal{W}_{1}(\nu_1, \nu_2)=\inf_{(Y_1, Y_2)\colon \mathcal{L}(Y_1)=\nu_1,\mathcal{L}(Y_2)=\nu_2}E\big(\|Y_1-Y_2\|_1\big),$$ and describes how the error decays to $0$ exponentially as $n$ goes to infinity. 

\begin{Proposition}\label{prop9} Suppose Assumption \ref{as0} and \ref{asergodic*} hold.  Furthermore, assume that $\cls=\cls_P$, that Lemma~\ref{lem:norm_1_implications} holds with $\alpha=1$,  and that $\Gamma$ is not a singleton set. Then for $n\ge 1$,
\begin{enumerate}[(a)]
\item $E(\|V_{n}^{*}\|_{1})<\infty, \quad M=\int_{\mathbb{R}^{d}_{\ge 0}}\|u\|_{1}\mu_{x}(du)<\infty, \quad r=-\log(E(\|C_0^x\|_1))>0$;
\item $\mathcal{W}_{1}(\mathcal{L}(V^{*}_n), \mu_{x})\leq  M e^{-rn}$. 
\end{enumerate} 
\end{Proposition}
\begin{proof}  
 {Consider a copy of $W_\infty^x$ which is independent of $U_x$, and recall that the distribution of $e^{A(x)U_{x}}W_\infty^x+G^{x}(U_{x})$ is $\mu_x$. Since $E(\|C_0^x\|_1)<1$ by Lemma~\ref{lem:norm_1_implications}(a) ($\alpha=1$), in order to prove part (a) we only need to show that $e^{A(x)U_{x}}V_{n}+G^{x}(U_{x})$ and $e^{A(x)U_{x}}W_\infty^x+G^{x}(U_{x})$ are integrable.} As $U_{x}\bigCI (V_{n},W_\infty^x)$, $E(G^{x}(U_{x}))<\infty$  and $\|e^{A(x)U_{x}}\|_{1}\le 1$ (using similar arguments  to those of the proof  of Lemma \ref{lem:interpretation}(c)) the assertion holds if both $W_\infty^x$ and $V_{n}$ are integrable. {The former is true by Lemma~\ref{lem:R}. For the latter,} from the recursive definition of $V_n$ we have
$$V_n=\sum_{j=1}^n \Big(\prod_{i=j+1}^n C_{0,i}^x\Big) D^x_{0,j}.$$
 Since Lemma \ref{lem:norm_1_implications} holds for $\alpha=1$, one has $E(\|C_{0,i}^x\|_{1})<1$ and $E(\|D_{0,i}^x\|_{1})<\infty$. Consequently, for $n\in\Z_{\ge 0}$,
\begin{equation*}
E(\|V_{n}\|_{1}) \le\sum_{j=1}^{n}E(\|D_{0,1}^{x}\|_{1})E(\|C_{0,1}^{x}\|_{1})^{j-1}\le \frac{E(\|D_{0}^{x}\|_{1})}{1 -E(\|C_{0}^{x}\|_{1})}<\infty.
\end{equation*}

For part (b), extend the sequence $(C_{0,i}^x, D_{0,i}^x)_{i=1,\ldots,n}$ to an infinite sequence of i.i.d. random variables, independent of $U_x$. Then, define
$$V=\sum_{j=1}^\infty \Big(\prod_{i=j-1}^1 C_{0,i}^x\Big) D^x_{0,j}\,\sim\, W_\infty^x$$
and let $V^*=e^{A(x)U_{x}}V+G^{x}(U_{x})$, which is distributed as $\mu_x$. Moreover, as shown in the proof of Lemma \ref{lem:R}, by exchangeability
$$V_n\sim \widetilde V_n=\sum_{j=1}^n \Big(\prod_{i=j-1}^1 C_{0,i}^x\Big) D^x_{0,j}.$$
  Then we have
\begin{align*}
\mathcal{W}_{1}(\mathcal{L}(V^{*}_n), \mu_{x})&\le E(\|e^{A(x)U_{x}}V+G^{x}(U_{x})-e^{A(x)U_{x}}\widetilde V_n-G^{x}(U_{x})\|_1)\\
&\le E(\|e^{A(x)U_{x}}\|_{1})E(\|V-\widetilde V_n\|_{1})\le E(\|V-\widetilde V_n\|_{1}).
\end{align*}
The proof is then concluded by observing that
$$V-\widetilde V_n=\sum_{j=n+1}^\infty \Big(\prod_{i=j-1}^1 C_{0,i}^x\Big) D^x_{0,j}=\Big(\prod_{i=n}^{1} C_{0,n-i}^x\Big)\widehat V,\qquad \widehat V\sim V.$$
\end{proof}

\begin{Example}
Consider the following reaction network with species $S_{1},S_{2},S_3$ and reactions
$$S_1+S_3 \ce{->[\lambda_1]}S_1,\quad  S_2\ce{->[\lambda_2]} S_2+S_3$$
$$0 \ce{->[\alpha_1]} S_1+2S_2\ce{<=>[\alpha_2][\alpha_3]}2S_1+S_2  \ce{->[\alpha_4]} 0,$$
taken with stochastic mass-action kinetics.
We take the   top line of reactions to be a stochastic reaction network with stochastic environment given by the the bottom line of the reactions, and rate functions,
\begin{equation}
\lambda_1(x,z)=\kappa_1x_1z,\quad \lambda_2(x,z)=\kappa_2x_2, \non 
\end{equation}\begin{equation}
\alpha_1(x)=a_1,\quad \alpha_2(x)=a_2x_1x_2^2,\quad \alpha_3(x)=a_3x_1^2x_2,\quad \alpha_4(x)=a_4x_1^2x_2, \non
\end{equation}
where $x=(x_1,x_2)$ denotes the counts of $S_1,S_2$, and $z$ the counts of $S_3$.
The stationary distribution of the environment is 
$$\pi(x_{1},x_{2})\propto \frac{b_{1}^{x_{1}}}{x_{1}!}\frac{b_{2}^{x_{2}}}{x_{2}!}\s\text{with}\s b_{1}=\bigg(\frac{a_{1}a_{2}}{a_{3}a_{4}}\bigg)^{\!\!\frac{1}{3}},\,\, b_{2}=\bigg(\frac{a_{1}a_{3}^{2}}{a_{4}a^{2}_{2}}\bigg)^{\!\!\frac{1}{3}}$$
on an irreducible component $\Gamma\subset \mathbb{Z}^{2}_{\ge 0}$ determined by $x_{1}+x_{2}$ mod $3=k,\, k=0,1,2$ \cite{anderson2010product}.

 Assumption \ref{as0} and \ref{asergodic*} hold, and Lemma \ref{lem:norm_1_implications} holds with $\alpha=1$.
   From Theorem \ref{T1}, the conditional distribution becomes
\begin{equation*}
\widehat\pi(z|x_{1},x_{2})= \int_{\R_{+}}P(\textup{Pois}(y)=z)\mu_{x}(dy)=\int_{\R_{+}}\frac{y^z}{z!}e^{-y}\mu_{x}(dy) 
\end{equation*}
with $x=(x_{1},x_{2}).$   Furthermore,  $\mu_{x}=\mathcal{L}(e^{-\kappa_1x_{1}U_x }V+G_x(U_x))$, where $U_x$ is an exponential random variable with rate $q_x$ that is independent of $V$, such that
\begin{eqnarray}
V&\sim &e^{-\int_{\tau_{0}^{x}}^{\tau_{1}^{j}}\kappa_1 X_{1}(s)ds}V+ \int_{\tau_{0}^{j}}^{\tau_{1}^{x}}\kappa_2X_{2}(u)e^{-\int_{u}^{\tau_{1}^{x}}\kappa_1X_{1}(s)ds}\,du, \non\\
q_x&=& \alpha_1(x)+\alpha_2(x)+\alpha_3(x)+\alpha_4(x).\non
\end{eqnarray}
 Using the iterative method of  Section \ref{sample} with  rate constants $(\kappa_1,\kappa_2,a_1,a_2,a_3,a_4)=(3,4,1,2,8,8)$, the realisations of $\widehat\pi(z|x_{1},x_{2})$ are found in Figure \ref{fig:one} for two different values of $x=(x_{1},x_{2})$ (with $K$ being the iteration number as in Proposition \ref{prop9}). The mixture of the Poisson distribution   changes with $(x_1,x_2)$.
 
\begin{figure}[h]
\centering
\includegraphics[width=14cm, height=7cm]{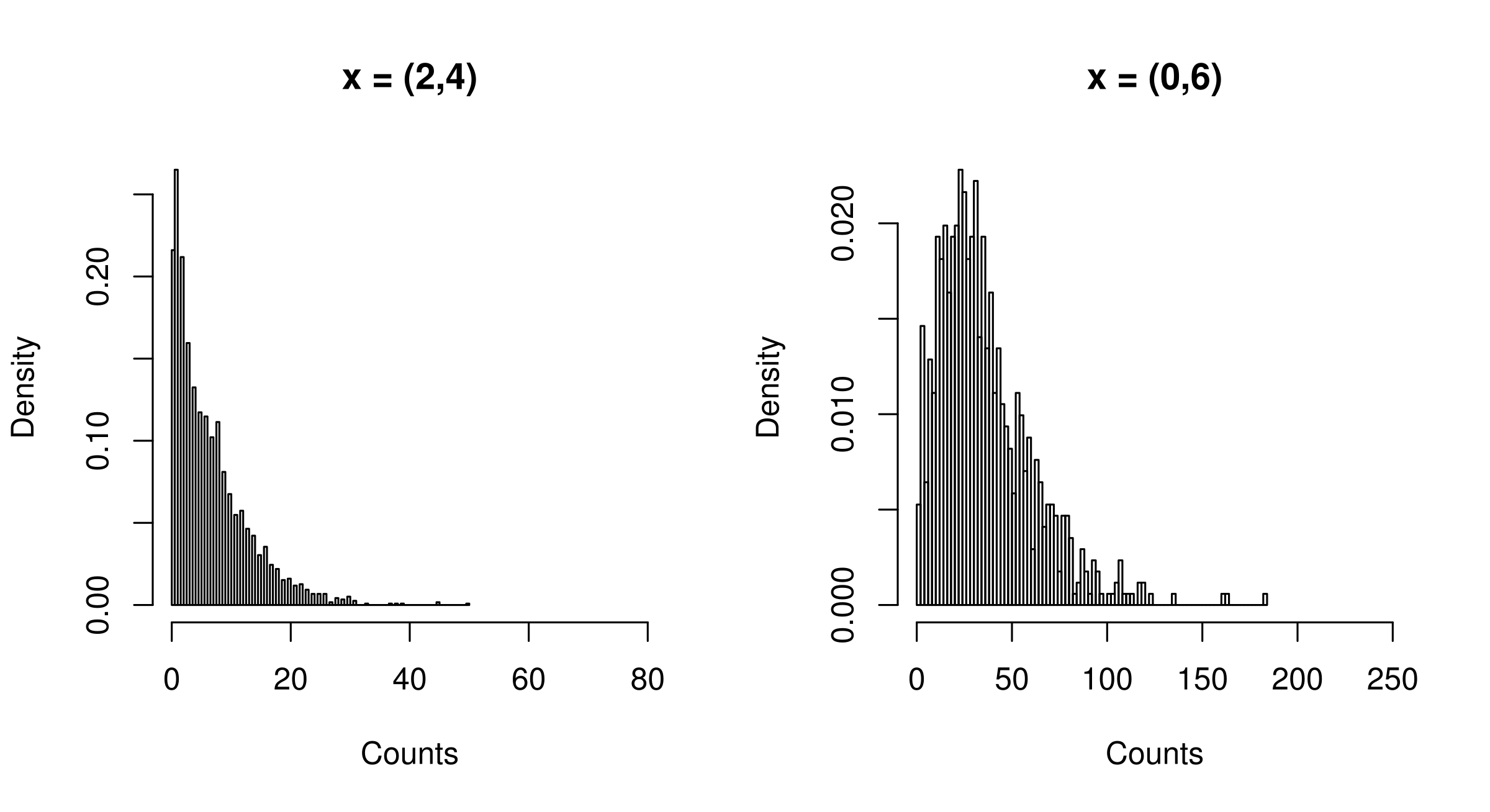}
\caption{Simulated densities with $K=20$ for two different states of the environment.  
}\label{fig:one}
\end{figure}
\end{Example}

\section{Gene Regulatory Networks}\label{regulatory}

An important class of reaction networks is gene regulatory networks \cite{peccoud1995markovian, kepler2001stochasticity, hornos2005self} that are building blocks of   large biological networks and underlie many cellular processes \cite{paulsson2004summing}. Gene regulation is  the process by which a gene is regulated by other molecules, generally known as transcription factors (TFs). Regulation controls the  production of the gene product, a protein. An elementary version of a gene regulatory system allows the gene to be in two states, an activated state where protein is produced and an inactivated state where protein is not produced (or produced at a lower rate).  The activation and inactivation is controlled by TFs with potential feedback mechanisms. Here we consider two   models of a gene regulatory mechanism \cite{peccoud1995markovian,jkedrak2016influence} and show that the stationary distributions exists and can be exactly characterised.   

The first is the following  reaction network with stochastic mass-action,
$$G\ce{<=>[\alpha_1][\alpha_2]} G',  \qquad G'\ce{->[\lambda'_1]} P+G',\qquad P\ce{->[\lambda_2]} \emptyset,$$
$$\alpha_1(x)=a_1x_1,\quad \alpha_2(x)=a_2x_2,\quad \lambda'_1(x,z)=\kappa'_1x_2,\quad\lambda_2(x,z)=\kappa_2z,$$
where  the dynamics of $G,G'$ constitutes the environment,  $x=(x_1,x_2)$ denotes the counts of $G,G'$ and $z$ that of $P$ \cite{peccoud1995markovian}. Here, $G$ denotes the inactive gene and $G'$ the activated gene. While in state $G'$ the gene  produces protein $P$. Hence, we have a stochastic reaction network with a stochastic environment described as
$$0\ce{->[\lambda'_1]} P,\qquad P\ce{->[\lambda_2]} \emptyset.$$

Denote the evolution of the counts of $G,G'$ by $((X_{1}(t),X_{2}(t))\colon t \ge 0) $ and observe that the total count  is conserved  $X_1(t)+X_2(t)=N$ for all $t\ge 0$. In \cite{peccoud1995markovian}, the authors consider $N=1$ and found the stationary distribution  using Kummer's differential equation after a suitable change of variable. Biologically,  $N=1$ would be the case for haploid (or monoploid) cells with one chromosome, such as bacteria. However, many organisms are diploid ($N=2$, like humans) or polyploid ($N\ge 2$, like many plants). Also, in particular in synthetic biology or experimental system biology, one might consider the situation in which the gene is located on a plasmid, a small chromosomal-like structure, with $N\ge 1$. If $N>1$, then the framework of \cite{peccoud1995markovian} will not work to characterise the stationary distribution.

To characterise the stationary distribution for general $N$ we rely on Theorem \ref{T1}. The stationary distribution of the counts of $G,G'$ is  a binomial distribution \cite{anderson2010product},
$$\pi_N(x)=\binom{N}{x_1}\!\left(\frac{a_2}{a_1+a_2}\right)^{\!\!x_1}\left(\frac{a_1}{a_1+a_2}\right)^{\!\!x_2}, \quad x\in \Gamma_N=\{(x_1,x_2)\in\Z^2_{\ge 0}\colon x_1+x_2=N\},$$
The function $G_x(s)$ is
$$G_x(s)= \sum_{n=0}^\infty\frac{1}{(n+1)!}(-\kappa_2)^n\kappa'_1(N-x_1)s^{n+1}=\frac{\kappa'_1(N-x_1)}{\kappa_2}(1-e^{-\kappa_2 s})\quad\text{for}\quad s\ge0.$$
Let $U_{x}$ be the exponential holding time in $x$ with rate $q_x=a_1x_1+a_2(N-x_1)$. 
Then the stationary distribution $\widehat\pi_N(x,z)$ is 
\begin{eqnarray}
\widehat\pi_N(x,z) &= \pi_N(x)\int_{\R_{\ge 0}}P(\textup{Pois}(y)= z) \mu_x(dy), \quad(x,z)\in \Gamma\times\Z_{\geq0}, \label{ExampleGRNs}
\end{eqnarray}
where $\mu_x$ is characterised in Theorem \ref{T1} as the law of $V^{*}= e^{-\kappa_2 U_x} V +G_{x}(U_{x}),$ 
and $V$ solves an explicit SRE.

The second example is a generalisation of the first and models a more complex mechanism adding a TF as a separate species \cite{jkedrak2016influence}.
 Specifically, the system is
$$G\ce{<=>[\alpha_1][\alpha_2]} G',  \qquad G\ce{->[\lambda_1]} P+G  \qquad G'\ce{->[\lambda'_1]} P+G',\qquad P\ce{->[\lambda_2]} \emptyset,$$
$$P\ce{->[\lambda_3]} P+TF, \qquad  TF\ce{->[\lambda_4]}\emptyset,$$
$$\alpha_1(x)=a_1x_1,\quad \alpha_2(x)=a_2x_2,\quad \lambda_1(x,y)=\kappa_1x_2,\quad \lambda'_1(x,y)=\kappa'_1x_2,$$
$$\lambda_2(x,y)=\kappa_2y,\quad \lambda_3(x,y,z)=\kappa_3y,\quad \lambda_4(x,y,z)=\kappa_4z,$$
 where $x=(x_1,x_2)$ denotes the counts of $G,G'$  $y$ the count of $P$ and $z$  that of $TF$, and the rate functions dependence of $x,y,z$ reflect how we will build up the network.
 
In this example, the gene product  $P$ can be produced from either version of the gene, but at different rates $\lambda_1,\lambda'_1$.  Subsequently, $P$ is involved in the production of a TF. The TF does not affect the activation of the gene nor the production of $P$. A still more complete model would consider a feedback mechanism, where the TF binds to the gene to enhance production of $P$.

As before we might at first take the dynamics of $G,G'$ to constitute the stochastic environment and 
$$0\ce{->[\lambda_1+\lambda'_1]} P,\qquad P\ce{->[\lambda_3]} \emptyset$$ 
to be a stochastic reaction network with stochastic environment.  Using Theorem \ref{T1} we  characterise  the stationary distribution, but now $\mu_x$ in \eqref{ExampleGRNs} is the law of
$$ e^{-\kappa_2 U_x} V_1 +\frac{\kappa'_1N+x_1(\kappa_1-\kappa'_1)}{\kappa_2}(1-e^{-\kappa_2 U_x} ), $$ 
where the law of $V_1$ is the unique solution of an explicit SRE, and $U_x\sim\textup{Exp}(\alpha_1(x)+\alpha_2(x))$.

At a second level, to understand the evolution of the TF, we might take the joint dynamics of $G,G',P$  represented by $((X_{1}(t),X_{2}(t),Y(t))\colon t\ge 0)$ to constitute the stochastic environment, where $Y(t)$ is the count of $P$, and
$$0\ce{->[\lambda_3]} TF,\qquad TF\ce{->[\lambda_4]} \emptyset$$
to be a stochastic reaction network with stochastic environment.   Using \eqref{invariant} the conditional probability is
$$\widehat\pi(z|x,y) = \int_{\R_{\ge 0}}P(\textup{Pois}(u)=z)\mu_{(x,y)}(du).$$ 
where $\mu_{(x,y)}$ is the law of 
$$e^{-\kappa_4 U_{(x,y)}}V_{2} +\frac{y\kappa_3}{\kappa_4}\left(1-e^{-\kappa_4 U_{(x,y)}}\right),$$
$U_{(x,y)}$ is an exponential variable with rate $q_x=\alpha_1(x)+\alpha_2(x)+\lambda_1(x,y)+\lambda'_1(x,y)+\lambda_2(x,y)$, and  $V_2$ solves an explicit SRE.

\section{Discussion}

A class of stochastic reaction networks with stochastic environment is studied and their finite time distribution as well as their closed form stationary distribution is characterised (when it exists). In applications, the stochastic environment itself is a stochastic reaction network. We focused on mono-molecular reaction networks, networks where reaction rates depend linearly on the counts of the species. This had the advantage that the finite time distribution of the molecular counts can be characterised through the paths of the stochastic environment \cite{jahnke2007solving}. The long term dynamics can then be characterised using the finite time characterisation. To go beyond the mono-molecular cases, finite time distributions are generally not known, perhaps with  \cite{laurenzi2000analytical} as the exception where the reversible reaction $A+B\ce{<=>} C$ is studied. However the method of \cite{laurenzi2000analytical} cannot be extended to the present case.

The linearity assumption plays an important role  in controlling non-explosivity of the molecular counts.  If  reactions are not mono-molecular, then it appears  we have less control of the behaviour of the system. In fact, we relied on the characterisation in \cite{jahnke2007solving} which in term builds on a close relationship between the stochastic system and a corresponding deterministic (ODE) system. If the reaction network is not mono-molecular, then this relationship between stochastic and deterministic counterparts is generally lost. See also Example \ref{ex:transient} and \ref{ex:explosion}.

 We end by listing some open directions for future research.

Consider feedback imposed by the protein in the gene activation step, assuming stochastic mass-action kinetics, 
$$G+P\ce{->[\alpha_1]} G'+P, \qquad G' \ce{->[\alpha_2]} G, \qquad G' \ce{->[\lambda'_1]}P+G', \qquad P\ce{->[\lambda_2]} \emptyset,$$
with  $\alpha_1(x)=a_1x_1$.
Clearly, we cannot take $G,G'$ to constitute the environment as the transformation of $G$ to $G'$ is catalysed by $P$. The evolution of the counts of $P$ can be stochastically upper-bounded by the reaction network
 $$\emptyset  \ce{->[\widetilde\alpha_1]} P,\qquad P\ce{->[\alpha]}\emptyset,$$
 where $\widetilde\alpha_1(x)=Na_1$ and $N$ is the conserved amount of $G,G'$. Hence the feedback network is  ergodic. More involved frameworks are needed to  characterise the stationary distribution in such cases.

A consequence of the Markov assumption \eqref{parasite} is reflected in following identity 
\begin{eqnarray*}
P(X(t)=x\big| (X(s),Z(s))=(x',z'))= P(X(t)=x\big| X(s)=x'),\quad 0<s<t,
\end{eqnarray*}
 that significantly simplified the structure of the stationary distribution here. In \cite{bowsher2010stochastic},   a reduction method for arbitrary reaction networks is considered based on a locally independence structure derived from a graph (the so-called Kinetic Independence Graph). The locally independence structure  is a weaker assumption than   assumption \eqref{parasite}. It would be interesting to characterise stochastic stability of reaction networks having independence structures weaker than \eqref{parasite}.

Generating a sample from $\mu_x$ will be extremely cumbersome if the conditioning event $\{X=x\}$ is rare as any  excursion path $(\tau_{0}^{x},\tau_{1}^{x}]$ will be large. In such cases, other simulation techniques are required. Here importance sampling might be applicable. It would be of interest to develop techniques that improves sampling for the models considered here.

\section{Acknowledgement}
CW is supported by the Independent Research Fund Denmark. DC has received funding from the European Research Council (ERC) under the European Union's Horizon 2020 research and innovation programme grant agreement No. 743269 (CyberGenetics project).

\section{Appendix}

\subsection{Proof of Lemma \ref{lem:R} }

Since $\tau^x_0=0$, then $C^x_{-1}=0$ and $D^x_{-1}=0$. By \eqref{eq:rel_C_Phi} and \eqref{eq:rel_D_W}, 
\begin{align*}
 \big(\Phi(\tau^x(t)),W(\tau^x(t))\big)&=\left(\prod_{k=0}^{n_t^x-1}C^x_k\,,\,\sum_{k=0}^{n_t^x-1} \left(\prod_{i=k+1}^{n_t^x-1} C^x_i\right)D^x_k\right)\\
 &=\left(C^x_{n_t^x-1}C^x_{n_t^x-2}\cdots C^x_0\,,\, D^x_{n_t^x-1}+C^x_{n_t^x-1}D^x_{n_t^x-2}+\dots+\left(\prod_{i=0}^{n_t^x-1} C^x_i\right)D^x_0\right).
\end{align*}
It follows from the renewal property of $\prX$ that 
$\big((C^x_0,D^x_0), (C^x_1,D^x_1), \dots, (C^x_{n_t^x-1},D^x_{n_t^x-1})\big)$
is distributed as
$\big((C^x_{n_t^x-1},D^x_{n_t^x-1}), (C^x_{n_t^x-2},D^x_{n_t^x-2}), \dots, (C^x_0,D^x_0)\big)$ for all $t\ge 0$. Hence,
\begin{equation}\label{eq:cond_equality}
 \big(\Phi(\tau^x(t)),W(\tau^x(t))\big)\,\sim\,\big(\widehat{\Phi}^x(t),\widehat{W}^x(t)\big),
 \end{equation}
 where
 \begin{align*}
 \widehat{\Phi}^x(t)&=\prod_{k=n_t^x-1}^{0}C^x_{k}=C^x_0C^x_1\cdots C^x_{n_t^x-1},\\
 \widehat{W}^x(t)&=\sum_{k=0}^{n_t^x-1} \left(\prod_{i=k-1}^{0} C^x_i\right) D^x_k
=D^x_0+C^x_0D^x_1+C^x_0C^x_1D^x_2+\dots+\left(\prod_{i=n_t^x-2}^{0} C^x_i\right) D^x_{n_t^x-1}
\end{align*}
are functions of $\prXuptot$. 
By the strong Markov property, \eqref{eq:cond_equality} also holds conditional on the event  $\{X(t)=x\}$.

We  will show that $(\widehat{\Phi}^x(t),\widehat{W}^x(t))$ converges strongly to a random variable $(\widehat\Phi^x_\infty, \widehat W^x_\infty)$ as $t\to\infty$, by proving it separately for the two entries, thus proving \eqref{eq:def_Phi_W_infty} and \eqref{eq:goal}.

 By assumption, $\widehat\Phi(t)$ is a block matrix with one block $\widehat\Phi^x_i(t)$, $1\leq i\leq h$, for each strongly connected component $\cls_i$ of the partition of $\cls$, and one block for $\cls_P$. Let $\alpha$ be as in Lemma~\ref{lem:norm_1_implications}\ref{item:irreducibility}.  Then with positive probability the matrices $\widehat\Phi^x_i(\tau^x_{\alpha})$, $1\leq i\leq h$, have all  entries positive. Hence, by Lemma~\ref{lem:interpretation} and \cite[Theorem 1.4]{infinite_matrix_product}, {for a fixed $1\leq i\leq h$} the limit (note the transpose form compared to $\widehat\Phi^x_i(t)$)
\begin{equation*}
M_\infty^x=\lim_{n\to\infty}M_{n,\alpha}^x=\lim_{n\to\infty}
\prod_{j=0}^n (C^x_{i,j \alpha+\alpha-1})^\top (C^x_{i,j \alpha+\alpha-2})^\top \dots (C^x_{i,j\alpha})^\top 
 \end{equation*}
exists a.s, where $C_{i,j}^x=\Phi_i(\tau^x_j,\tau^x_{j+1})$. Likewise, the limit  
$$M^x_\infty(\ell)=
\left(\prod_{j=0}^\infty (C^x_{i,j \alpha+\alpha-1+\ell})^\top (C^x_{i,j \alpha+\alpha-2+\ell})^\top \dots (C^x_{i,j \alpha+\ell})^\top \right)(C^x_{i,\ell-1})^\top(C^x_{i,\ell-2})^\top\dots (C^x_{i,0})^\top$$
exists a.s.\! for $1\leq \ell\leq\alpha-1$.
In order to prove that $\widehat\Phi^x_i(t)$ converges strongly, we only need to show that $M^x_\infty(\ell)=M^x_\infty$ a.s. for all $1\leq \ell\leq\alpha-1$. If it was not so, $\lim_{n\to\infty} M_{n,\alpha+1}$ would not exist. However, it converges by the same arguments as before.  Hence, $\widehat{\Phi}_i^x(t)$ converges strongly, and so does $\widehat{\Phi}^x(t)$.  It also follows from \cite[Theorem 1.4]{infinite_matrix_product} that in the limit the columns are identical. 

We now study the convergence of $\widehat{W}^x(t)$ and show the limit has finite  expectation. By Lemma~\ref{lem:interpretation}\ref{item:probability}, all entries of the matrices $C^x_n$ are non-negative. The same holds for the entries of the matrices $D^x_n$. Finally, for $0\leq t_1\leq t_2$ we have $n_{t_1}^x\leq n_{t_2}^x$, therefore
\begin{equation*}
 \widehat{W}^x(t_1)\leq \widehat{W}^x(t_2)\quad\text{for all}\quad 0\le t_1<t_2,
\end{equation*}
where the inequality holds component-wise. Hence, $\lim_{t\to\infty}\widehat W^x(t)$ exists a.s. and
\begin{equation}\label{egheghelrjgelrkjh}
 \lim_{t\to\infty}\widehat W^x(t)=W^x_\infty.
\end{equation}
Equation \eqref{eq:goal} is then proved.  In order to prove $E(W^x_\infty)<\infty$, define 
$J\in\Z^{d\times d}_{\geq0}$ as the diagonal matrix with entries
$$J_{ij}=\begin{cases}
          1 & \text{if }i=j\text{ and }S_i\text{ is properly degraded}, \\
          0 & \text{otherwise}.
         \end{cases}$$
Since $(B_X)_i=0$ for any $1\le i\le d$ such that $0\ntransf S_i$ then the corresponding columns of $D^x_k$, $k\ge 0$, are zero. Consequently, we might write  $C^x_k J$ in the definition of $W^x_\infty$ rather than $C^x_k$, as the zero columns progresses through the product,
$$W^x_\infty=\sum_{k=0}^{\infty} \left(\prod_{i=0}^{k-1} C^x_iJ\right)  D^x_k.$$
By monotone convergence and the strong Markov property,
$$E(\|W^x_\infty\|_1)\leq\sum_{k=0}^{\infty} \left(\prod_{i=0}^{k-1} E(\| C^x_iJ\|_1)\right) E(\|D^x_k\|_1)=\sum_{k=0}^{\infty} \left(E(\| C^x_0J\|_1)\right)^k E(\| D^x_0\|_1).$$
Finally, due to Lemma~\ref{lem:norm_1_implications}, we have
$E(\| C^x_0J\|_1)<1\quad\text{and}\quad E(\| D^x_0\|_1)<\infty,$
 implying
$$E(\|W^x_\infty\|_1)\leq\frac{E(\|D^x_0\|_1)}{1-E(\| C^x_0J\|_1)}<\infty.$$

 In order to prove  \eqref{eq:conditional_limit}, note that for $\eta>0$ and any function $\sigma\colon\R_{\geq0}\to \R_{\geq0}$ with $\lim_{t\to\infty} \sigma(t)=\infty$, 
\begin{equation}\label{eqeqeqeqeqeqeq}
 \lim_{t\to\infty}P\Big(\Big\|\Big(\widehat\Phi^x(\sigma(t)),\widehat W^x(\sigma(t))\Big)-\Big(\Phi^x_\infty,W_\infty^x\Big)\Big\|_1>\eta \,\Big\vert\, X(t)=x\Big)=0.
\end{equation}
Indeed, by ergodicity of $\prX$ there exists $T_\eta$ such that
$$P(X(t)=x)>\pi(x)-\eta\quad\text{for all }t\geq T_\eta,$$
and if $\eta$ is smaller than $\pi(x)$
\begin{align*}
 0&<(\pi(x)-\eta)P\Big(\Big\|\Big(\widehat\Phi^x(\sigma(t)),\widehat W^x(\sigma(t))\Big)-\Big(\Phi^x_\infty,W_\infty^x\Big)\Big\|_1>\eta \Big\vert X(t)=x\Big)\\
 &<P\Big(\Big\|\Big(\widehat\Phi^x(\sigma(t)),\widehat W^x(\sigma(t))\Big)-\Big(\Phi^x_\infty,W_\infty^x\Big)\Big\|_1>\eta , X(t)=x\Big)\\
 &\leq P\Big(\Big\|\Big(\widehat\Phi^x(\sigma(t)),\widehat W^x(\sigma(t))\Big)-\Big(\Phi^x_\infty,W_\infty^x\Big)\Big\|_1>\eta\Big).
\end{align*}
Equation \eqref{eqeqeqeqeqeqeq} then follows from the strong convergence of $(\widehat\Phi^x(\sigma(t)), \widehat W^x(\sigma(t)))$. Therefore, for any continuity set $A$ of $(\Phi^x_\infty,W^x_\infty)$ we have
\begin{equation}\label{eq:sdghskdh}
\begin{split}
 \lim_{t\to \infty}P\Big(\Big(\Phi(\tau^x(t)), W(\tau^x(t))\Big)\in A\,\Big\vert\,X(t)=x\Big)&=
 \lim_{t\to \infty}P\Big(\Big(\widehat\Phi(t), \widehat W^x(t)\Big)\in A\,\Big\vert\,X(t)=x\Big)\\
 &=\lim_{t\to \infty}P\Big(\Big(\Phi^x_\infty,W_\infty^x\Big)\in A\,\Big\vert\,X(t)=x\Big)\\
 &=\lim_{t\to \infty}P\Big(\Big(\widehat\Phi(\sigma(t)), \widehat W^x(\sigma(t))\Big)\in A\,\Big\vert\,X(t)=x\Big),
 \end{split}
\end{equation}
where the first equality follows from \eqref{eq:cond_equality} and \eqref{egheghelrjgelrkjh}, the second equality follows from \eqref{eqeqeqeqeqeqeq} with the function $\sigma$ being the identity, and the third equality follows from \eqref{eqeqeqeqeqeqeq} for a general function $\sigma$ diverging to $\infty$. Note that \eqref{eq:conditional_limit} is proven if in the last line of \eqref{eqeqeqeqeqeqeq} the conditioning part is removed. The idea is to show that the conditioning part do not play any role in the limit by using the arbitrariness of the function $\sigma$ and the ergodicity of $\prX$.

Fix $0<\varepsilon<\pi(x)$, and let $\Gamma_\varepsilon\subseteq\Gamma$ be a finite set of states such that
\begin{equation}\label{eq:gamma_epsilon}
 \sum_{x'\notin\Gamma_\varepsilon}\pi(x')<\frac{\varepsilon}{2},
\end{equation}
where we use the convention that a sum over an empty set is zero. For any real numbers $0\leq s\leq t$ define
$$\theta_x(s,t)=\sup_{x'\in\Gamma_\varepsilon} \Big|P(X(t)=x\,|\,X(s)=x')-P(X(t)=x)\Big|.$$
Since $\Gamma_\varepsilon$ is a finite set, and by ergodicity of the CTMC $\prX$ with the fixed initial condition $X(0)=x$, we have that for all $s\geq0$
\begin{equation}\label{eq:key_equation}
 \lim_{t\to \infty} \theta_x(s,t)=0.
\end{equation}
A first consequence of \eqref{eq:key_equation} is that there exists $T_\varepsilon$ such that {$\theta_x(0,t)<\varepsilon$ for all $t\geq T_\varepsilon$. Hence,} we can define the function $\sigma\colon\R_{\geq0}\to\R_{\geq0}$ as
$$\sigma(t)=\begin{cases}
        0 & \text{if }t\le T_\varepsilon, \\
    {   \sup\{0\le s\le t\colon\, \theta_x(s,t)<\varepsilon\} } & \text{if }t> T_\varepsilon.
       \end{cases}$$
It also follows from \eqref{eq:key_equation} that $\sigma(t)$ goes to infinity as $t$ goes to infinity. 
Hence, it follows from the ergodicity of $\prX$ that by potentially increasing $T_\varepsilon$ we might further assume 
\begin{align}
\label{eq:st_infinity}
 \sum_{x'\in\Gamma}\Big|P(X(\sigma(t))=x')-\pi(x')\Big|&<\frac{\varepsilon}{2}\quad\text{for all}\quad t>T_\varepsilon,\\
 \label{eq:t_infinity}
 \sum_{x'\in\Gamma}\Big|P(X(t)=x')-\pi(x')\Big|&<\varepsilon\quad\text{for all}\quad t>T_\varepsilon.
\end{align}
In particular, it follows from \eqref{eq:gamma_epsilon} and \eqref{eq:st_infinity} that
\begin{equation}\label{eq:bound}
{0<\sum_{x'\notin\Gamma_\varepsilon}P(X(\sigma(t))=x')<\varepsilon\quad\text{for all}\quad t>T_\varepsilon. }
\end{equation}
For all $t>T_\varepsilon$, we have
\begin{align}\label{eq:lower}
 &P\Big(\Big(\widehat{\Phi}^x(\sigma(t)),\widehat W^x(\sigma(t))\Big)\in A\,\Big\vert\,X(t)=x\Big) \non \\
 &\quad\geq\sum_{x'\in\Gamma_\varepsilon}P\Big(\Big(\widehat{\Phi}^x(\sigma(t)),\widehat W^x(\sigma(t))\Big)\in A\,\Big\vert\,X(\sigma(t))=x', X(t)=x\Big)P(X(\sigma(t))=x'\,|\, X(t)=x)  \non \\
 &\quad=\sum_{x'\in\Gamma_\varepsilon}P\Big(\Big(\widehat{\Phi}^x(\sigma(t)),\widehat W^x(\sigma(t))\Big)\in A\,\Big\vert\,X(\sigma(t))=x'\Big)\frac{P(X(t)=x\,|\, X(\sigma(t))=x')P(X(\sigma(t))=x')}{P( X(t)=x)} \non \\
 &\quad\geq\frac{\pi(x)-2\varepsilon}{\pi(x)+\varepsilon}\sum_{x'\in\Gamma_\varepsilon}P\Big(\Big(\widehat{\Phi}^x(\sigma(t)),\widehat W^x(\sigma(t))\Big)\in A\,\Big\vert\,X(\sigma(t))=x'\Big)P(X(\sigma(t))=x') \non \\
 &\quad\geq\frac{\pi(x)-2\varepsilon}{\pi(x)+\varepsilon}\Big(P\Big(\Big(\widehat{\Phi}^x(\sigma(t)),\widehat W^x(\sigma(t))\Big)\in A\Big)-\varepsilon\Big),
\end{align}
where in the second line we conditioned on the possible values of $X(\sigma(t))$, in the  third line we utilise {the fact that $\sigma(t)\leq t$ by definition,} the Markov property of $\prX$, and Bayes' formula. In the forth line we use the definition of $\sigma(t)$ and \eqref{eq:t_infinity}, and in the last line we use \eqref{eq:bound}. Similarly, we have
\begin{equation}\label{eq:upper}
\begin{split}
 &P\Big(\Big(\widehat{\Phi}^x(\sigma(t)),\widehat W^x(\sigma(t))\Big)\in A\,\Big\vert\,X(t)=x\Big)\\
 &\quad\leq\sum_{x'\in\Gamma_\varepsilon}P\Big(\Big(\widehat{\Phi}^x(\sigma(t)),\widehat W^x(\sigma(t))\Big)\in A\,\Big\vert\,X(\sigma(t))=x', X(t)=x\Big)P(X(\sigma(t))=x'\,|\, X(t)=x)\\
 &\quad\quad+\sum_{x'\notin\Gamma_\varepsilon}P(X(\sigma(t))=x'\,|\, X(t)=x)\\
 &\quad=\sum_{x'\in\Gamma_\varepsilon}P\Big(\Big(\widehat{\Phi}^x(\sigma(t)),\widehat W^x(\sigma(t))\Big)\in A\,\Big\vert\,X(\sigma(t))=x'\Big)\frac{P(X(t)=x\,|\, X(\sigma(t))=x')P(X(\sigma(t))=x')}{P( X(t)=x)}\\
  &\quad\quad+\sum_{x'\notin\Gamma_\varepsilon}\frac{P(X(t)=x\,|\, X(\sigma(t))=x')P(X(\sigma(t))=x')}{P( X(t)=x)}\\
 &\quad\leq\frac{\pi(x)+2\varepsilon}{\pi(x)-\varepsilon}\left(\varepsilon+\sum_{x'\in\Gamma_\varepsilon}P\Big(\Big(\widehat{\Phi}^x(\sigma(t)),\widehat W^x(\sigma(t))\Big)\in A\,\Big\vert\,X(\sigma(t))=x'\Big)P(X(\sigma(t))=x')\right)\\
 &\quad\leq\frac{\pi(x)+2\varepsilon}{\pi(x)-\varepsilon}\Big(P\Big(\Big(\widehat{\Phi}^x(\sigma(t)),\widehat W^x(\sigma(t))\Big)\in A\Big)+\varepsilon\Big),
\end{split}
\end{equation}
where for the first inequality we condition on the possible values of $X(\sigma(t))$, for the equality afterwards we utilise {the fact that $\sigma(t)\leq t$, }the Markov property of $\prX$ and Bayes' formula. For the consecutive inequality we use the definition of $\sigma(t)$, \eqref{eq:t_infinity} and \eqref{eq:bound}, and in the last line we use the law of total probability. In conclusion, it follows from \eqref{eq:lower} and \eqref{eq:upper} {and by the arbitrariness of $0<\varepsilon<\pi(x)$} that
$$\lim_{t\to\infty} P\Big(\Big(\widehat{\Phi}^x(\sigma(t)),\widehat W^x(\sigma(t))\Big)\in A\,\Big\vert\,X(t)=x\Big)=\lim_{t\to\infty} P\Big(\Big(\widehat{\Phi}^x(\sigma(t)),\widehat W^x(\sigma(t))\Big)\in A\Big)$$
which in turn implies \eqref{eq:conditional_limit} by \eqref{egheghelrjgelrkjh} and \eqref{eq:sdghskdh}. The proof is then completed.
\hfill$\square$
\subsection{Proof of Lemma \ref{lem:uniqueness}}

Recall that $(C,D)$ is distributed as $(C_0^x,D_0^x)$. Choose a version of $W^x_\infty$ such that $(C,D)$ is independent of $(C_i^x,D_i^x)$, $i\in\Z_{\ge 0}$. Then it follows straightforwardly from the definition of {$\Phi^x_\infty$ and} $W^x_\infty$ in \eqref{eq:def_Phi_W_infty} that {$(\Phi^x_\infty,W^x_\infty)$} fulfils the SRE.

Using \cite{erhardsson2014conditions}[Theorem 2.1] the uniqueness of the {solution to $V_2\sim CV_2+D$} can be deduced 
 under the conditions
\begin{equation}
\prod_{i=0}^{n}C _{i}\ce{->[a.s.]} 0, \s\text{and}\s
\left(\prod_{i=0}^{n-1}C_{i}\right)D_n \ce{->[a.s.]} 0 \s\textup{as }\quad n \to \infty,\label{defstochrec2*}
\end{equation}
where $(C_{n},D_{n})_{n\ge 0}$  are i.i.d.~distributed as $(C,D)\sim (C^x_0,D^x_0)$. 
According to \eqref{eq:def_C} and Lemma \ref{lem:norm_1_implications}(a), there is $\alpha\ge 1$ such that $E(\| \widetilde C_i\|_1) <1$, where $\widetilde C_i=C_{\alpha-1+i\alpha}\ldots C_{1+i\alpha}C_{i\alpha}$, $i\ge 0$. Note that for $1\le k\le \alpha$ and $n\ge 0$,
$$M_{k+n\alpha}=\left\|\left(\prod_{i=0}^{k-1+n\alpha}C_{i}\right)\!D_{k+n\alpha}\right\|_1\le
\left(\prod_{i=0}^{n-1} \left\| \widetilde C_i\right\|_1 \right)\left\| D_{k+n\alpha}\right\|_1.$$
 Applying Jensen's inequality,  $E(\log\| \widetilde C_i\|_1) \le \log E(\| \widetilde C_i\|_1) <0$. Using \eqref{eq:def_C} 
  and  Lemma \ref{lem:norm_1_implications}(b), we similarly find $E(\log\|D_n\|_1)\le \log E(\| D_n \|_1)<\infty$.   Hence as a consequence of the strong law of large numbers,
  $$\limsup_{k+n\alpha\to\infty} \frac{1}{k+n\alpha}\log(M_{k+n\alpha})\le 
\limsup_{n\to\infty}\left(\frac{1}{n} \sum_{i=0}^{n-1}\log\| { \widetilde C_i} \|_1\right)+ \limsup_{n\to\infty}\left(\frac{1}{n} \log\| D_n \|_1\right)\,<\,0\quad \text{a.s.},$$
proving the second assertion of \eqref{defstochrec2*} (the second term converges to $0$ if $E(\log\| D_n \|_1)$ is finite and is otherwise $\le0$). The first assertion {of \eqref{defstochrec2*}} follows similarly.

\hfill$\square$

\subsection{Proof of Lemma \ref{cor:ergodic}}

We state a preliminary lemma first.

\begin{Lemma}\label{lem:tight_poisson}
 Consider a stochastic process $\{L(t)\colon t\ge 0\}$ on $\R_{\geq0}$, and let $m$ be a positive real number. Let $\{J(t)\colon t\ge 0\}$ be another stochastic process, such that $J(t)\sim \textup{Pois}(m L(t))$ for any $t\ge 0$. Then, $\{L(t)\colon t\ge 0\}$ is tight if and only if $\{J(t)\colon t\ge 0\}$ is tight.
\end{Lemma}
\begin{proof}
 Assume that $\{L(t)\colon t\ge 0\}$ is tight. Then, by definition, for all $\varepsilon>0$ there exists $M_\varepsilon$ such that
 $$\sup_{t\ge 0}P(L(t)> M_\varepsilon)\leq \varepsilon.$$
For  $0<\varepsilon\leq 1$  define
 $$\widetilde{M}_\varepsilon=\frac{2m}{\varepsilon}\left(1-\frac{\varepsilon}{2}\right)M_{\frac{\varepsilon}{2}}.$$
We have
 \begin{align*}
  \sup_{t\ge 0}P\Big(J(t)> \widetilde{M}_\varepsilon\Big)&\leq \left(1-\frac{\varepsilon}{2}\right)P\Big(\textup{Pois}(m M_{\frac{\varepsilon}{2}})> \widetilde{M}_\varepsilon\Big)+\frac{\varepsilon}{2}\\
  &\leq \left(1-\frac{\varepsilon}{2}\right)\frac{m M_{\frac{\varepsilon}{2}}}{\widetilde{M}_\varepsilon}+\frac{\varepsilon}{2}=\varepsilon,
 \end{align*}
 where   the second inequality follows from Markov's inequality. Hence, the process $\{J(t)\colon t\ge 0\}$  is tight.
 
 For the other direction, assume that $\{L(t)\colon t\ge 0\}$  is not tight. Then, there exists $\varepsilon>0$ such that for all $M>0$ there exists $t\ge 0$ with $P(L(t)> M)> \varepsilon$. Then, for all $M> \frac{4}{m}$ there exists $t\ge 0$ with
 \begin{align*}
  P\Big(J(t)> mM-2\sqrt{mM}\Big)&> \varepsilon P\Big(\textup{Pois}(mM)> mM-2\sqrt{mM}\Big)\\
  &> \varepsilon P\Big(|\textup{Pois}(mM)-mM|\leq 2\sqrt{mM}\Big)\geq \frac{\varepsilon}{4},
 \end{align*}
 where we used Chebychev's inequality in the last step. Therefore, the process  $\{J(t)\colon t\ge 0\}$ is neither tight, and the proof is concluded.
\end{proof}

\begin{Lemma}\label{lem:tight_iff_tight}
 The process $\prW$ is tight if and only if the process $\prZ$ is tight.
\end{Lemma}
\begin{proof}
First, it follows from the definition of tightness that $\prW$ is tight if and only if $\{\|W(t)\|_1\colon t\ge 0\}$ is tight. By Lemma~\ref{lem:interpretation}\ref{item:probability}, we further have
$$\|W(t)\|_1=\left\|\int_{0}^{t}\Phi(u,t)B_X(u)\,du\right\|_1=\int_{0}^{t}e^\top\Phi(u,t)B_X(u)\,du.$$
 Assume that $\prW$ is tight. Since $Z(t)\in\Z_{\geq0}^d$, we have $\|Z(t)\|_1=\sum_{j=1}^{d} Z_j(t)=e^\top Z(t)$. Using Proposition~\ref{P2}, the random variable $\|Z(t)\|_1$, conditioned on $\mathcal{F}^X_t$,  is stochastically bounded by
\begin{equation}\label{eq:Ztbound}
\|Z(t)\|_1\preceq \|Z(0)\|_1 +m\sum_{j=1}^d\sum_{\nu \in \Theta_{m_j,d}} N_{\nu j}(t),
\end{equation}
where $m=\max(m_j\colon j=1,\ldots,d)$ and $N_{\nu j}$ are independent random variables with
$$N_{\nu j}(t)\sim \textup{Pois}\left(\int_{0}^{t}\lambda_{0j}(X(u))\,g^X_{u,t}(\nu, m_{j})\,du\right).$$
By definition of configuration and of $m$, we have $\|\nu\|_1\le m$ for all $\nu \in \cup_{j=1}^d\Theta_{m_j,d}$. Using the definition of $g^X_{u,t}(\nu, m_j)$ and independence (conditioned on $\Ft$) of the random variables $N_{\nu j}(t)$ for $\nu\in \Theta_{m_j,d}$, $1\leq j\leq d$,  equation \eqref{eq:Ztbound}  reduces to 
\begin{align*}
 \|Z(t)\|_1&\preceq \|Z(0)\|_1 +m\textup{Pois}\left(\sum_{j=1}^d  \int_{0}^{t} \lambda_{0j}(X(u))\left[1-g^X_{u,t}((0,\ldots,0), m_{j})\right]\,du\right)\\
 &\sim \|Z(0)\|_1 +m\textup{Pois}\left(\sum_{j=1}^d   \int_{0}^{t}\lambda_{0j}(X(u))\left[1-(1-e^\top\Phi(u,t)e_j)^{m_j}\right]\,du\right)\\
 &\preceq  \|Z(0)\|_1 +m\textup{Pois}\left(m \sum_{j=1}^d\int_{0}^{t}\lambda_{0j}(X(u))\,e^\top\Phi(u,t)e_j\,du\right)\\
&\sim \|Z(0)\|_1 +m\textup{Pois} \left(m \int_{0}^{t}e^\top\Phi(u,t)B_X(u)\,du\right),
 \end{align*}
by recalling the definition \eqref{eq:BX} of $B_X(u)$. Furthermore, it is  used that $1-ka\le (1-a)^k$ for $0\le a\le 1$ and $k\geq 1$. Using Lemma~\ref{lem:tight_poisson} and the stochastic upper bound, tightness of $\prZ$ follows. 

Conversely, assume that $\prZ$ is tight with the given $m$. Then, the molecular count process $\prZp$ corresponding to the process with  $m'=(\min(1,m_1),\ldots,\min(1,m_{d}))$ in \eqref{model} is also tight. Indeed, $\|Z'(t)\|_1$ is stochastically bounded from above by $\|Z(t)\|_1$ for all $t\ge 0$, because
\begin{itemize}
 \item the lifetime distribution of any molecule created in any of the two settings is the same as it does not depend on the choice of $m$,
 \item every molecule that is created in any of the two settings evolves independently of the other molecules being present, given $\{\Ft\}_{t\ge 0}$ (see Remark~\ref{rem:independence}),
 \item in the process $\prZp$ less molecules are created than in the process $\prZ$ (the two processes can be coupled such that the reactions $0\to m_jS_j$ and $0\to m'_jS_j$ occur at the same time).
\end{itemize}

The choice of $m$ does  not affect the quantities $A_X(t)$ and $B_X(t)$, hence the process $\prW$ is not affected either. It follows from Proposition~\ref{P2} that
\begin{align*}
 \textup{Pois} \left(\int_{0}^{t}e^\top\Phi(u,t)B_X(u)\,du\right)&\sim e^\top \sum_{j=1}^d\sum_{\nu\in \Theta_{m'_j,d}}\nu N'_{\nu j}(t)\\
 &\preceq e^\top Z'(t) =\|Z'(t)\|_1.
\end{align*}
Finally, it follows from the tightness of $\prZp$ and  Lemma~\ref{lem:tight_poisson} that $\prW$ is tight, which concludes the proof.
\end{proof}

We are now ready to prove Lemma \ref{cor:ergodic}.

\begin{proof}[Proof of Lemma \ref{cor:ergodic}]
Assume that $\prXZ$ is ergodic for any initial condition. Then $\prZ$ is tight, which by Lemma \ref{lem:tight_iff_tight} implies that $\prW$ is tight. 

Now assume that $\prW$ is tight. Then, by Lemma \ref{lem:tight_iff_tight} the process $\prZ$ is tight as well. It follows that $\prXZ$ is tight because $\prX$ is ergodic by assumption. Hence, by Prokhorov’s theorem, $\prXZ$ is sequentially compact \cite{billingsley}. Since $\prXZ$ is a CTMC, the accumulation point of $\prXZ$ is unique, which implies that $\prXZ$ is ergodic.
\end{proof}

\subsection{Additional results}

\begin{Lemma}\label{Appe1}
Suppose  $\Phi(t)$ is a fundamental matrix solution tp $x'=A(t)x$. 
The solution of  $\lambda'(t)=A(t)\lambda(t)+b(t)$ with initial condition $\lambda(u)=\lambda_{u}$,  is 
 \begin{equation} \lambda(t)=\Phi(t)\Phi^{-1}(u)\lambda_{u}+\int_{u}^{t}\Phi(t)\Phi^{-1}(s)b(s)ds,\quad t>u\ge 0. \label{Leme1}
\end{equation}
\end{Lemma}
\begin{proof}
Using that $\Phi'(t)=A(t) \Phi(t)$ by assumption, we find the derivative of \eqref{Leme1} to be
\begin{align*}
\lambda'(t)&=\Phi(t)'\Phi^{-1}(u)\lambda_{u}+\Phi(t)\Phi^{-1}(t)b(t)+ \int_{u}^{t}\Phi'(t)\Phi^{-1}(s)b(s)ds\\
&=A(t)\Phi(t)\Phi^{-1}(u)\lambda_u+b(t)+\int_{u}^{t}A(t)\Phi(t)\Phi^{-1}(s)b(s)ds=A(t)\lambda(t)+b(t).
\end{align*}
\end{proof}

 %\bibliographystyle{plain}
 %\bibliography{bibparasite_v2}

\end{document}